\documentclass[10pt,a4paper,reqno]{article} 
\usepackage{mysty}

\usepackage[thinlines]{easytable}

\usepackage{fullpage}
\usepackage{multirow}
 
\usepackage{subfig}

\usepackage{patchcmd}
\makeatletter
\patchcommand\@starttoc{\begin{quote}}{\end{quote}}
\makeatother

\SetLabelAlign{parright}{\parbox[t]{\labelwidth}{\raggedleft{#1}}}
\setlist[description]{style=multiline,topsep=4pt,align=parright}

\graphicspath{{./figure/}}

\makeatletter
\let\reftagform@=\tagform@
\def\tagform@#1{\maketag@@@{(\ignorespaces\textcolor{black}{#1}\unskip\@@italiccorr)}}
\newcommand{\iref}[1]{\textup{\reftagform@{\tcr{\ref{#1}}}}}
\makeatother

\newcommand*{\mcap}{\mathbin{\scalebox{1.5}{\ensuremath{\cap}}}}%

\usepackage[noindentafter]{titlesec}
\usepackage{bold-extra}

\titlespacing\section{0pt}{11pt plus 4pt minus 2pt}{6pt plus 2pt minus 2pt}
\titlespacing\subsection{0pt}{10pt plus 4pt minus 2pt}{4pt plus 2pt minus 2pt}
\titlespacing\subsubsection{0pt}{8pt plus 4pt minus 2pt}{4pt plus 2pt minus 2pt}
\titlespacing\paragraph{0pt}{6pt plus 4pt minus 2pt}{6pt plus 2pt minus 2pt}

\begin{document}

	\setlength{\abovedisplayskip}{4.5pt}
	\setlength{\belowdisplayskip}{4pt}

\title{The Fun is Finite: Douglas--Rachford and Sudoku Puzzle --- Finite Termination and Local Linear Convergence}
\author{{Robert Tovey\thanks{INRIA Paris, France. E-mail: robert.tovey@inria.fr}~~~~~~~and~~~~
		Jingwei Liang\thanks{Queen Mary University of London, UK. E-mail: jl993@cam.ac.uk.}} 
		}
\date{}
\maketitle

\begin{abstract}
In recent years, the Douglas--Rachford splitting method has been shown to be effective at solving many non-convex optimization problems. In this paper we present a local convergence analysis for non-convex feasibility problems and show that both finite termination and local linear convergence are obtained.
For a generalization of the Sudoku puzzle, we prove that the local linear rate of convergence of Douglas--Rachford is exactly $\frac{\sqrt{5}}{5}$ and independent of puzzle size. For the $s$-queens problem we prove that Douglas--Rachford converges after a finite number of iterations.
Numerical results on solving Sudoku puzzles and $s$-queens puzzles are provided to support our theoretical findings. 
\end{abstract}

\begin{keywords}
Douglas--Rachford $\cdot$ Feasibility problem $\cdot$ Sudoku Puzzle $\cdot$ Finite Termination \and Local Linear Convergence
\end{keywords}

\begin{AMS}
{49J52 $\cdot$ 65K05 $\cdot$ 65K10 $\cdot$ 90C25}
\end{AMS}

\section{Introduction}
%
%


Given two non-empty sets $C$ and $S$ whose intersection is also non-empty, the feasibility problem aims to find a common point in the intersection $C\cap S$. 
In the literature, popular numerical schemes for solving feasibility problems are developed based on projection, among them alternating projection is the fundamental one. 
The method of alternating projection was first introduced by von Neumann for the case of two linear subspaces \cite{von1933functional}, then was extended to closed convex sets by Bregman \cite{bregman1965method}. 
Relaxation is a standard approach to speed up alternating projection and related work can be found in \cite{luke2004relaxed,cegielski2008relaxed}.

Proximal splitting methods, such as Forward--Backward \cite{lions1979splitting} splitting and Backward--Backward splitting \cite{combettes2011proximal}, and Peaceman--Rachford/Douglas--Rachford splitting \cite{peaceman1955numerical,douglas1956numerical}, can also be applied to solve feasibility problem either directly or up to reformulation. Moreover, equivalence between projection based methods and proximal splitting methods can be established, such as alternating projection is equivalent to Backward--Backward splitting while relaxed alternating relaxed projection covers Peaceman--Rachford/Douglas--Rachford splitting as special cases \cite{combettes2011proximal}.

Our focus in this paper is Douglas-Rachford splitting method, which has shown to be effective for solving feasibility problem, particularly in the non-convex setting \cite{bauschke2017finite}.
However, the convergence property is rather less understood than its convex counter part. One reason for this is that Douglas--Rachford splitting method is not symmetric and non-descent, when compared to (proximal) gradient descent whose non-convex case is much better studied \cite{attouch2010proximal}.
Research on non-convex Douglas--Rachford either focuses on specific cases or imposing stronger assumptions (\eg smoothness) and proposes modifications to the original iteration. For instance \cite{artacho2013global} considers Douglas--Rachford splitting for solving feasibility problem of a line intersecting with a circle, and conditions for convergence are provided. In \cite{li2016douglas}, the authors proposed a damped Douglas--Rachford splitting method for general non-convex optimization problem under the condition that one function has a Lipschitz continuous gradient.

The study of this paper is motivated by applying Douglas--Rachford to solve Sudoku puzzle\footnote{\url{https://en.wikipedia.org/wiki/Sudoku}}, for which three different convergence behaviors are observed
\begin{itemize}
\item Globally, the method converges {sub-linearly}. 
\item Locally, two regimes occur: finite termination and linear convergence. 
\end{itemize}
Finite termination and local linear convergence are reported in the literature \cite{BauschkeFiniteDR15,bauschke2017finite}, however, conditions in respective work either are designed for convex setting or cannot be satisfied by Sudoku puzzle. 
Therefore, a new analysis is needed for Douglas--Rachford splitting which is the aim of this paper: 
\begin{enumerate}
\item {\bf Finite termination} 
Under a non-degeneracy condition, see \eqref{eq:ndc}, 
we show in Section \ref{sec:local_rate} that one sequence generated by Douglas--Rachford splitting has the finite termination property. 
All sequences terminate in a finite number of iterations if the problem satisfies certain assumptions (\eg polyhedrality, see Assumptions \iref{A:C}-\iref{A:intersection}).

\item {\bf Local linear convergence} 
We also provide a precise characterization for the local linear convergence of Douglas--Rachford splitting method. Particularly, for Sudoku puzzle, we prove that locally the linear rate of convergence of Douglas--Rachford splitting method is precisely $\frac{\sqrt{5}}{5}$. Moreover, such a rate is independent of puzzle size.  
For the damped Douglas--Rachford splitting method, we also provide an exact estimation of the local linear rate which depends on the damping coefficient. 

\end{enumerate}


\paragraph{Relation to Prior Work} \label{subsec:rel}
There are several existing work studying the finite termination property of the standard Douglas--Rachford splitting method. In \cite{BauschkeFiniteDR15}, the authors established finite convergence of Douglas--Rachford in the presence of Slater's condition, for solving convex feasibility problems where one set is an affine subspace and the other is a polyhedron, or one set is an epigraph and the other one is a hyperplane. 
The result was extended to general convex optimization problems in \cite{liang2017local} under the notion of partial smoothness \cite{Lewis-PartlySmooth}. 
In \cite{matsushita2016finite}, finite termination is proved for finding a point which is guaranteed to be in the interior of one set whose interior is assumed to be non-empty. The result of \cite{BauschkeFiniteDR15} was later extended to the non-convex case in \cite{bauschke2017finite}, where one of the two sets can be finite. 


For local linear convergence, results can be found in for instance \cite{phan2016linear} where linear convergence of Douglas--Rachford splitting method is established under a regularity condition. Similar results can be found in \cite{HesseLuke13,hesse2013nonconvex}. 
Under a constraint qualification condition, \cite{li2016douglas} also discussed the local linear convergence property of the damped Douglas--Rachford splitting method.


\paragraph{Paper Organization}
The rest of the paper is organized as follows. 
Some preliminaries are collected in Section~\ref{sec:preliminary}. Section~\ref{sec:drs} states our main assumptions on problem~\eqref{eq:feasibility} and introduces the standard and damped Douglas--Rachford algorithms, global convergence is also discussed. Our main result on local convergence of Douglas--Rachford is presented in Section~\ref{sec:local_rate}. In Section~\ref{sec:experiment}, we report numerical experiments on Sudoku puzzle and $s$-queens puzzle to support our theoretical findings.

\section{Preliminaries}\label{sec:preliminary}

Throughout the paper, $\bbN$ is the set of nonnegative integers, $\bbR^{n}$ is a finite $n$-dimensional real Euclidean space equipped with scalar product $\iprod{\cdot}{\cdot}$ and norm $\norm{\cdot}$. $\Id$ denotes the identity operator on $\bbR^{n}$. 
For a matrix $M \in \bbR^{n \times n}$, we denote $\rho(M)$ its spectral radius.

\paragraph{Projection and reflection}
Below we collect necessary concepts related to sets. 

\begin{definition}[Distance and indicator function]
Let $C \subset \bbR^n$ be non-empty and $x \in \bbR^n$. The {\it distance function} of $x$ to $C$ is defined by
\[
\dist(x, C) \eqdef \inf_{y\in C} \norm{x - y}.
\]
The {\it indicator function} of $C$ is defined by $\iota_{C}(x)
=
\left\{
\begin{aligned}
0 &: x \in C , \\
\pinf &: x \notin C .
\end{aligned}
\right.
$
\end{definition}

\begin{definition}[Projection \& reflection]
Let $C \subset \bbR^n$ be non-empty and $x\in\bbR^n$. The projection of $x$ onto $C$, denoted by $\proj_{C}(x)$, is a set defined by 
\[
\proj_{C}(x) \eqdef \Ba{ y \in C : \norm{x-y} = \dist(x, C) } .
\]
The mapping $\proj_{C} \colon \bbR^n \setvalued C$ is called the {\it projection operator}. The {\it relaxed projection} $\proj^{\lambda}_{C}$ is defined via
\[
\proj^{\lambda}_{C}(x) = \lambda \proj_{C}(x) + (1-\lambda) x ,
\]
where $\lambda \in ]0, 2]$ is the relaxation parameter. 
When $\lambda = 2$, the corresponding mapping is called {\it reflection} and denoted by $\rproj_{C}(x) = 2 \proj_{C}(x) - x $. 
\end{definition}

\begin{definition}[Prox-regularity]
A non-empty closed set $C \subset \bbR^n$ is {\it prox-regular} at $x \in C$ for $v$ if $x = \proj_{C}(x+v)$.  If $\proj_{C}$ is single-valued in an open neighborhood of $x \in C$, $C$ is called {\it prox-regular} at $x$. 
\end{definition}

\begin{definition}[Normal vector]
Given $C \subset \bbR^n$ and $x \in C$, the {\it proximal normal cone} $\pnormal{C}(x)$ of $C$ at $x$ is defined by
\[
\pnormal{C}(x) = \mathrm{cone}\Pa{ \proj_{C}^{-1}(x) - x } .
\]
The {\it limiting normal cone} $\normal{C}(x)$ is defined as any vector that can be written as the limit of proximal normals: $v \in \normal{C}(x)$ if and only if there exists sequences $\seq{\xk} \in C$ and $\seq{\vk}$ in $\normal{C}(\xk)$ such that $\xk\to x $ and $\vk \to v$. 
\end{definition}

Let $C_1, C_2 \subset \bbR^n$ be two sets with non-empty intersection. The feasibility problem of $C_1, C_2$ is to find a common point in the intersection, \ie
\[
\find x \in \bbR^n \qstq x \in C_1 \mcap C_2 .
\]
A fundamental algorithm to solve the problem is the alternating projection method which, as indicated by the name, represents the procedure: from a given point $x_0$, apply projection onto each set alternatively
\beq\label{eq:ap}
\xkp = \proj_{C_2} \proj_{C_1} (\xk) .
\eeq
One can also consider relaxation for each projection operator and the whole iteration, which results in the following iteration
\[
\xkp = \xk + \lambda \Pa{ \proj^{\lambda_{2}}_{C_2} \proj^{\lambda_{1}}_{C_1} (\xk) - \xk }  ,
\]
where $\lambda, {\lambda_{1}}, {\lambda_{2}}$ are relaxation parameters. 
The iteration becomes Peaceman--Rachford splitting (alternating reflection) for $(\lambda, {\lambda_{1}}, {\lambda_{2}}) = (1,2,2)$ and Douglas--Rachford splitting for $(\lambda, {\lambda_{1}}, {\lambda_{2}}) = (1/2,2,2)$. We refer to \cite{bauschke1996projection} for a survey on the alternating projection method. 



%
%
%

\paragraph{Convergent Matrices} 
To discuss the local linear convergence, we need the following preliminary results on convergent matrices which are taken from \cite{meyer2000matrix}.

\begin{definition}[Convergent matrices]\label{def:convmat}
A matrix $M \in \bbR^{n \times n}$ is {\it convergent} to $\Minf \in \bbR^{n \times n}$ if, and only if, $\lim_{k\to\pinf}\norm{M^k - \Minf} = 0 $.
$M$ is said to be {\it linearly convergent} if there exists $\eta \in [0, 1[$ and $K \in \bbN$ such that for all $ k \geq K$, there holds $\norm{M^{k}-\Minf} = O(\eta^{k}) $. 
If $M$ does not converge at any rate $\eta'\in[0,\eta[$ then $\eta$ is called the {\it optimum convergence rate}.
\end{definition}

\begin{definition}[Semi-simple eigenvalue]\label{def:semi_simple}
For $M \in\bbR^{n\times n}$, an eigenvalue $\eta$ is called {\it semi-simple} if and only if $\rank(M-\eta\Id) = \rank((M-\eta\Id)^2)$. 
\end{definition}
%

%

\begin{theorem}[Limits of powers]\label{thm:limits-of-power}
For $M \in \bbR^{n\times n}$, the power of $M$ converges to $\Minf$ if and only if  $\rho(M) < 1$ or $\rho(M) = 1$ with $1$ being the only eigenvalue on the complex unit circle and semi-simple. 
%
\end{theorem}



Whenever $M$ is convergent, it converges {linearly} to $\Minf$, and we have the following lemma.

\begin{lemma}[Convergence rate]\label{lem:rate-M}
Suppose $M \in \bbR^{n\times n}$ is convergent to some $\Minf \in \bbR^{n\times n}$, then
\begin{enumerate}[label={\rm (\roman{*})}]
\item
for any $\kinN$, 
\[
M^{k} - \Minf = (M-\Minf)^{k} \qandq
\norm{M^{k} - \Minf} \leq \norm{M-\Minf}^{k} .
\]
The equality holds only when $M$ is {normal}.
\item
We have $\rho(M-\Minf) < 1$, and $M$ is linearly convergent for any $\eta \in ]\rho(M-\Minf), 1[$. 
%
\item
$\rho(M-\Minf)$ is the optimal convergence rate if one of the following holds
	\begin{enumerate}[label={\rm (\alph{*})}]
	\item
	$M$ is normal.
	\item
	All the eigenvalues $\eta \in \Theta_{M}$ such that $\abs{\eta} = \rho(M-\Minf)$ are semi-simple.
	\end{enumerate}
\end{enumerate}
\end{lemma}
\begin{proof}
See Theorems 2.12, 2.13, 2.15 and 2.16 of \cite{bauschke2016optimal}.
\end{proof}


\paragraph{Angles between Subspaces}

To precisely characterize the local linear convergence rate, we need the following concepts regarding the angles between subspaces. Let $T_1$ and $T_2$ be two linear subspaces with dimension $p\eqdef\dim(T_1)$ and $q\eqdef\dim(T_2)$, and without loss of generality, suppose that $1\leq p \leq q \leq n-1$.

\begin{definition}[Principal angles] 
	The {\it principal angles} $\theta_k \in [0,\frac{\pi}{2}]$, $k=1,\ldots,p$ between linear subspaces $T_1$ and $T_2$ are defined by, with $u_0 = v_0 \eqdef 0$ and inductively
	\begin{align*}
	\cos(\theta_k) \eqdef \dotp{u_k}{v_k} = \max \big\{\dotp{u}{v} ~~~\mathrm{s.t.}~~~	
	& u \in T_1, v \in T_2, \norm{u}=1, \norm{v}=1, \\
	&\dotp{u}{u_i}=\dotp{v}{v_i}=0, ~ i=0,\dotsm,k-1\big\} .
	\end{align*}
	The principal angles $\theta_k$ are unique with $0 \leq \theta_1 \leq \theta_2 \leq \dotsm \leq \theta_p \leq \pi/2$.
\end{definition}

\begin{definition}[Friedrichs angle] 
	The {\it Friedrichs angle} $\theta_{F} \in [0,\frac{\pi}{2}]$ between $T_1$ and $T_2$ is
	\[
	\begin{aligned}
	\cos\Pa{ \theta_F(T_1,T_2) } \eqdef \max \dotp{u}{v} ~~~\mathrm{s.t.}~~~
	&u \in T_1 \cap (T_1 \cap T_2)^\perp, \norm{u}=1 , \\
	&v \in T_2 \cap (T_1 \cap T_2)^\perp, \norm{v}=1  .
	\end{aligned}
	\]
\end{definition}


The following lemma shows the relation between the Friedrichs and principal angles.

\begin{lemma}[{\cite[Proposition~3.3]{bauschke2016optimal}}]
	\label{lem:fapa}
	We have $\theta_{F}(T_1,T_2) = \theta_{d+1} > 0$ where $d \eqdef \dim(T_1 \cap T_2)$.
\end{lemma}



\section{Problem and algorithm}
\label{sec:drs}

The formal statement of the feasibility problem is written below
\begin{equation}\label{eq:feasibility}
\find~~{x \in \RR^n} \qstq x \in C \mcap S  ,
\end{equation}
where the following assumptions are imposed
\begin{enumerate}[leftmargin=4em, label= ({\textbf{A.\arabic{*}})}, ref= \textbf{A.\arabic{*}}]
	\item \label{A:C} 
	$C \subset \bbR^n$ is a closed set; 
	\item \label{A:S} 
	$S \subset \bbR^n$ is an affine subspace;
	\item \label{A:intersection} 
	$C \mcap S \neq \emptyset$, \ie the intersection is non-empty.
\end{enumerate}
Note that the problem \eqref{eq:feasibility} is not necessarily convex as we suppose $C$ is only non-empty and closed. Examples of \eqref{eq:feasibility} are provided in Section \ref{sec:local_rate}, including the Sudoku puzzle and $s$-queens puzzle. 


\subsection{Douglas--Rachford splitting method}


The development of Douglas--Rachford (DR) splitting method \cite{douglas1956numerical} dates back to 1950s for solving numerical PDEs. In recently years, the method has also been shown to be effective for non-convex feasibility problem \cite{hesse2014alternating,artacho2014recent}. 
Details of the method for solving \eqref{eq:feasibility} is described in Algorithm~\ref{alg:dr}. 

\begin{center}
	\begin{minipage}{0.95\linewidth}
		\begin{algorithm}[H]
			\caption{Standard Douglas--Rachford splitting (DR)} \label{alg:dr}
			{\noindent{\bf{Initial}}}: $z_0 \in \bbR^n$\;
			\Repeat{convergence}{
				\beq\label{eq:dr}
                \begin{aligned}
                	\xkp &= \proj_{S} (\zk)  , \\
                	\ukp &\in \proj_{C}\pa{2\xkp - \zk} , \\
                	\zkp &= \zk + \ukp - \xkp  , 
                \end{aligned}
                \eeq
			}
		\end{algorithm}
	\end{minipage}
\end{center}

The above iteration can be written as the fixed-point iteration of variable $\zk$. Denote the fixed-point operator
\beq\label{eq:fpo-dr}
	\fDR \eqdef \sfrac{1}{2} \Pa{ \pa{2\proj_{C}-\Id} \pa{2\proj_{S} - \Id } + \Id } ,
\eeq
then we have $\zkp = \fDR(\zk)$. 
The other two variables $\uk, \xk$ are called the {\it shadow sequences} \cite{bauschke2014local}. 

Determining the convergence properties of Douglas--Rachford splitting for the non-convex setting is a challenging problem, the non-descent property of the method makes it much harder to obtain convergence result than the descent-type methods which includes (proximal) gradient descent \cite{attouch2010proximal}. 

Moreover, since the method has three different sequences $\uk,\xk$ and $\zk$, various different convergence behaviors may occur. We refer to \cite{bauschke2017finite} for more detailed discussions. 
In Example~\ref{exp:circle_line}, we demonstrate a case of a circle intersecting with a line where:
\begin{itemize}
	\item The shadow sequences $\seq{\uk}$, $\seq{\xk}$ converge to $\usol$ and $\xsol$, respectively. But $\usol\neq\xsol$. 
	\item The fixed-point sequence $\seq{\zk}$ diverges.
\end{itemize}
As our main interest in this paper is to study the local behavior, for the rest of the paper, we suppose that the standard DR is globally convergent: 
\begin{enumerate}[leftmargin=4em, label= ({\textbf{A.\arabic{*}})}, ref= \textbf{A.\arabic{*}}]
\setcounter{enumi}{3}
	\item \label{A:D} 
	The standard Douglas--Rachford splitting method for solving \eqref{eq:feasibility} is globally convergent. 
\end{enumerate}
Consequently, one has
\[
\zk \to \zsol \in \fix(\fDR) \eqdef \Ba{ z \in \bbR^n : z = \fDR(z) } 
\qandq
\uk, \xk \to \xsol \in \proj_{S}(\zsol) .
\]
To avoid assumption \eqref{A:D}, people either turn to specific cases \cite{artacho2016global} or imposing stronger assumptions such as smoothness \cite{themelis2020douglas}. 
Modifications to the original Douglas--Rachford splitting method are also considered in the literature.
Below we describe a damped version of Douglas--Rachford proposed in \cite{li2016douglas}.


Solving the feasibility problem \eqref{eq:feasibility} is equivalent to the following constrained smooth optimization. 
\beq\label{eq:dist_C}
\min_{x\in\bbR^n}~ \sfrac{1}{2} \dist^2(x, S) ~\qstq x \in C .
\eeq
In \eqref{eq:dr}, the update of $\xkp$ is equivalent to solving 
$\min_{x\in\bbR^n}\, \iota_{S}(x) + \frac{1}{2\gamma}\norm{x - \zk}^2$.
Replacing the indicator function with the distance function, 
\[
\min_{x\in\bbR^n}\, \sfrac{1}{2} \dist^2(x, S) + \sfrac{1}{2\gamma}\norm{x - \zk}^2 ,
\]
we then get
\[
\xkp 
= \sfrac{1}{1+\gamma} \Pa{{\zk} + \gamma \proj_{S} (\zk) }
= \zk + \sfrac{\gamma}{1+\gamma} \Pa{ \proj_{S} (\zk) - \zk }
= \proj_{S}^{\frac{\gamma}{1+\gamma}}(\zk)  .
\]
As a result, we obtain the algorithm proposed in \cite{li2016douglas}.

\begin{center}
	\begin{minipage}{0.95\linewidth}
		\begin{algorithm}[H]
			\caption{A damped Douglas--Rachford splitting (dDR)} \label{alg:ddr}
			{\noindent{\bf{Initial}}}: $\gamma > 0$, $z_0 \in \bbR^n$\; 
			\Repeat{convergence}{
				\beq\label{eq:ddr}
				\begin{aligned}
					\xkp &= \proj_{S}^{\frac{\gamma}{1+\gamma}}(\zk) , \\ 
					\ukp &\in \proj_{C}\pa{2\xkp - \zk} , \\
					\zkp &= \zk + \ukp - \xkp  , 
				\end{aligned}
				\eeq
			}
		\end{algorithm}
	\end{minipage}
\end{center}

We refer to the original work \cite{li2016douglas} for a more detailed discussion of Algorithm~\ref{alg:ddr}. 
%
When $\gamma = +\infty$, Algorithm~\ref{alg:ddr} recovers the standard Douglas--Rachford splitting method \eqref{eq:dr}. 
The fixed-point operator of dDR reads
\beq\label{eq:fpo-dr}
	\fdDR \eqdef \sfrac{1}{2} \Pa{ \pa{2\proj_{C}-\Id} \pa{2\proj_{S}^{\frac{\gamma}{1+\gamma}} - \Id } + \Id } .
\eeq
We have the following convergence result of dDR from \cite{li2016douglas}.

\begin{lemma}[{Global convergence of dDR \cite[Theorem 5]{li2016douglas}}]\label{thm:global_convergence}
For the non-convex feasibility problem \eqref{eq:feasibility}, suppose Assumptions \iref{A:C}-\iref{A:intersection} hold and moreover $C$ is compact. Choose $\gamma \in ]0, \sqrt{3/2}-1[$ for the Douglas--Rachford splitting method \eqref{eq:ddr}, then the sequence $\seq{\uk,\xk, \zk}$ is bounded, and given any cluster point $(\usol, \xsol, \zsol)$ of the sequence, there holds $\norm{\zk-\zkm}\to0$, $\usol = \xsol$ and $\xsol$ is a stationary point of the problem \eqref{eq:dist_C}. 
\end{lemma}

In the example below, we demonstrate a case where DR fails to solve the problem while dDR succeeds.

\begin{example}[A circle intersects with a line]\label{exp:circle_line}
Let $C = \ba{x \in \bbR^2 : \norm{x}=1}$ be the unit circle and $S = \ba{x \in \bbR^2 : \dotp{x}{\left(\begin{smallmatrix}1\\2\end{smallmatrix}\right)} = \sqrt{2}}$ be a line that intersects with $C$ at two different points. 
For both methods, same initial point $z_0 = (-10,-8)$ is chosen. 
%
For damped DR, we set $\gamma = \frac{1}{5}$. In Figure~\ref{fig:circle_line} we observe: 
\begin{itemize}
\item For the standard DR (left): $\zk$ is not convergent, $\uk$ and $\xk$ converge to two different points and the method fails to find a feasible point. 
\item For the damped DR (right): all three sequences converge to the same feasible point. 
\end{itemize}
We refer to \cite{artacho2013global} for a detailed discussion of the convergence properties of the standard DR for solving this feasibility problem.
\begin{figure}[!ht]
	\centering
	\subfloat[Standard Douglas--Rachford]{ \includegraphics[width=0.375\linewidth]{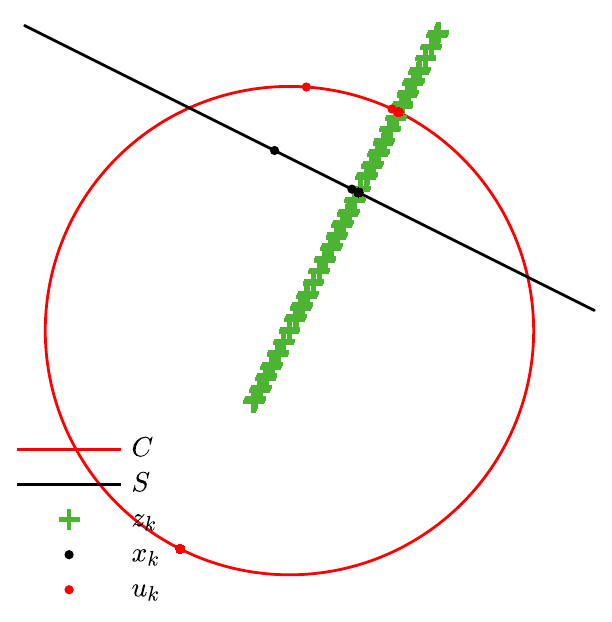} } \hspace{3ex}
	\subfloat[Damped Douglas--Rachford]{ \includegraphics[width=0.375\linewidth]{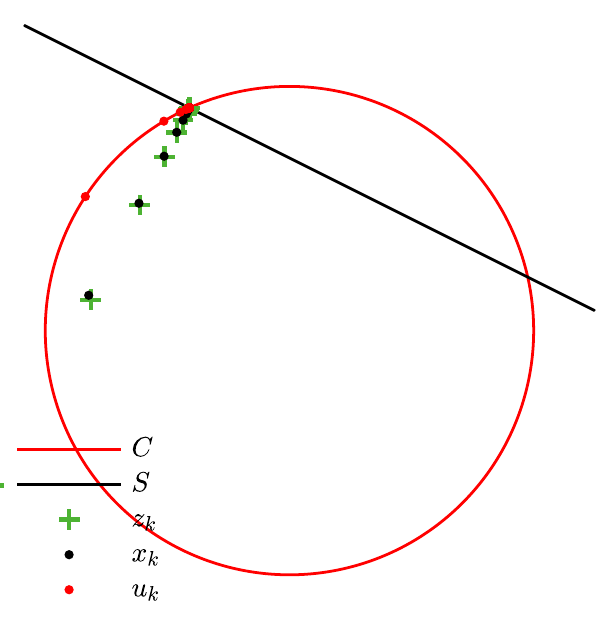} } \\
	\caption{Convergence behavior of the standard and damped Douglas--Rachford for solving the problem of a line intersecting with a circle.} 
	\label{fig:circle_line}
\end{figure}
\end{example}

Though dDR solves the problem, as remarked in the original paper \cite{li2016douglas}, it may also converges to some stationary point of \eqref{eq:dist_C} which is not a solution. 
In fact, as we shall see in the numerical experiments, for the Sudoku puzzle and $s$-queens puzzle, dDR with $\gamma = \frac{1}{5} \in ]0, \sqrt{3/2}-1[$ fails all tests while sDR achieves very good performance; see Table \ref{tb:cmp}.

\subsection{Problems with more than two sets}

Up to now, we have been dealing with feasibility problem of two sets, while in various scenarios we need to deal with the case of finding common points of more than two sets. In what follows, we briefly show that, by a {\it product space trick}, we can reformulate the problem into the form of \eqref{eq:feasibility}.

Let $m \geq 2$ be an integer, $C_i$ a non-empty closed set for each $i \in \ba{1,\ldots,m}$. Consider the following feasibility problem
\beq\label{eq:m_sets}
\find x~ \in \bbR^{n} \qstq  x \in \mcap_{i=1}^{m} C_i .
\eeq
Let $\bcH=\underset{\textrm{$m$ times}}{\underbrace{\bbR^{n} \times \cdots \times \bbR^{n}}}$ be the product space endowed with the scalar inner-product and norm $$ \forall \bmx, \bmy \in \bcH,~\bprod{\bmx}{\bmy} =\msum_{i=1}^m \iprod{x_i}{y_i},~\bnorm{\bmx} = \Pa{\ssum_{i=1}^m\norm{x_i}^2}^{1/2} .$$ 
Let $\bcC \eqdef C_{1} \times \dotsm \times C_{m}$, then $\bcC \subset \bcH$, and denote the subspace $\bcS \eqdef \ba{\bmx=(x_i)_i\in\bcH : x_1=\dotsm=x_m}$. 
The feasibility problem \eqref{eq:m_sets} can be reformulated into the following form
\beq\label{eq:msets}
\find \bmx\in\bcH \qstq \bmx \in  \bcC {\mcap} \bcS .
\eeq
The projection operator of $\bcC$ is component-wise for each set $C_i, i=1,\ldots,m$, 
\[
\proj_{\bcC}\bmx = (\proj_{C_1}x_1,\dotsm, \proj_{C_m}x_m).
\] 
Define $\bcK: \bbR^n \to \bcS,~x\mapsto(x,\dotsm,x) $, 
then we have $\proj_{\bcS}(\bmx) = \bcK \pa{\tfrac{1}{m}\sum_{i=1}^m x_i}$. 

Adapting the standard Douglas--Rachford to the case of \eqref{eq:m_sets}, we obtain 
\beq\label{eq:dr_msets}
\begin{aligned}
&x_{k+1} = \sfrac{1}{m}\msum_{i=1}^m z_{i,k}  , \\
& \textrm{For $i=1,\ldots,m$:} \\
&\left\lfloor 
\begin{aligned}
	u_{i,k+1} &\in \proj_{C_i} (2x_{k+1} - z_{i,k}) , \\
	z_{i,k+1} &= z_{i,k} + u_{i,k+1} - x_{k+1} .
\end{aligned}
\right. 
\end{aligned}
\eeq
Note that for the standard DR, there is no need to store $\bmx$ and simply $x_{k+1} = \frac{1}{m}\msum_{i=1}^m z_{i,k}$ is sufficient. 
Correspondingly, we also have the following iteration for the damped Douglas--Rachford splitting method:\beq\label{eq:ddr_msets}
\begin{aligned}
& \textrm{For $i=1,\ldots,m$:} \\
&\left\lfloor 
\begin{aligned}
	x_{i,k+1} &= \sfrac{1}{1+\gamma} \Pa{ z_{i,k} + \gamma \sfrac{1}{m}\msum_{j=1}^m z_{j,k} }, \\
    u_{i,k+1} &\in \proj_{C_i} (2x_{i,k+1} - z_{i,k}) , \\
    z_{i,k+1} &= z_{i,k} + u_{i,k+1} - x_{i,k+1} .
\end{aligned}
\right. 
\end{aligned}
\eeq

\section{Local convergence of Douglas--Rachford splitting}
\label{sec:local_rate}

In this section we present our main result, the local convergence of Douglas--Rachford splitting. 
We first present the result in a general setting and then specialize to the case of Sudoku and $s$-queens puzzles.

\paragraph{Non-degeneracy condition} To deliver the result, a non-degeneracy condition is needed for set $C$. 
Assume Assumption \iref{A:D} holds for standard DR and that dDR is ran under the condition of Lemma \ref{thm:global_convergence}, then at convergence for both methods we have $\zk\to\zsol$ and $\uk, \xk \to \xsol$. 
We assume that $C$ is prox-regular at $\xsol$ for $\xsol-\zsol$ and the following condition holds 
\beq\label{eq:ndc}
\xsol - \zsol \in \mathrm{int}\Pa{ \normal{C}\pa{\xsol} } 
\eeq
where $\mathrm{int}(\cdot)$ stands for the interior of the set. 

\begin{remark}
The non-degeneracy condition \eqref{eq:ndc} requires $\normal{C}\pa{\xsol}$ has a non-empty interior, which means that $\xsol$ is a vertex of the set $C$. 
A graphical illustration of the non-degeneracy condition \eqref{eq:ndc} is provided in Figure \ref{fig:normalcone} below. 
\end{remark}


\begin{figure}[!ht]
	\centering
	\includegraphics[width=0.45\linewidth]{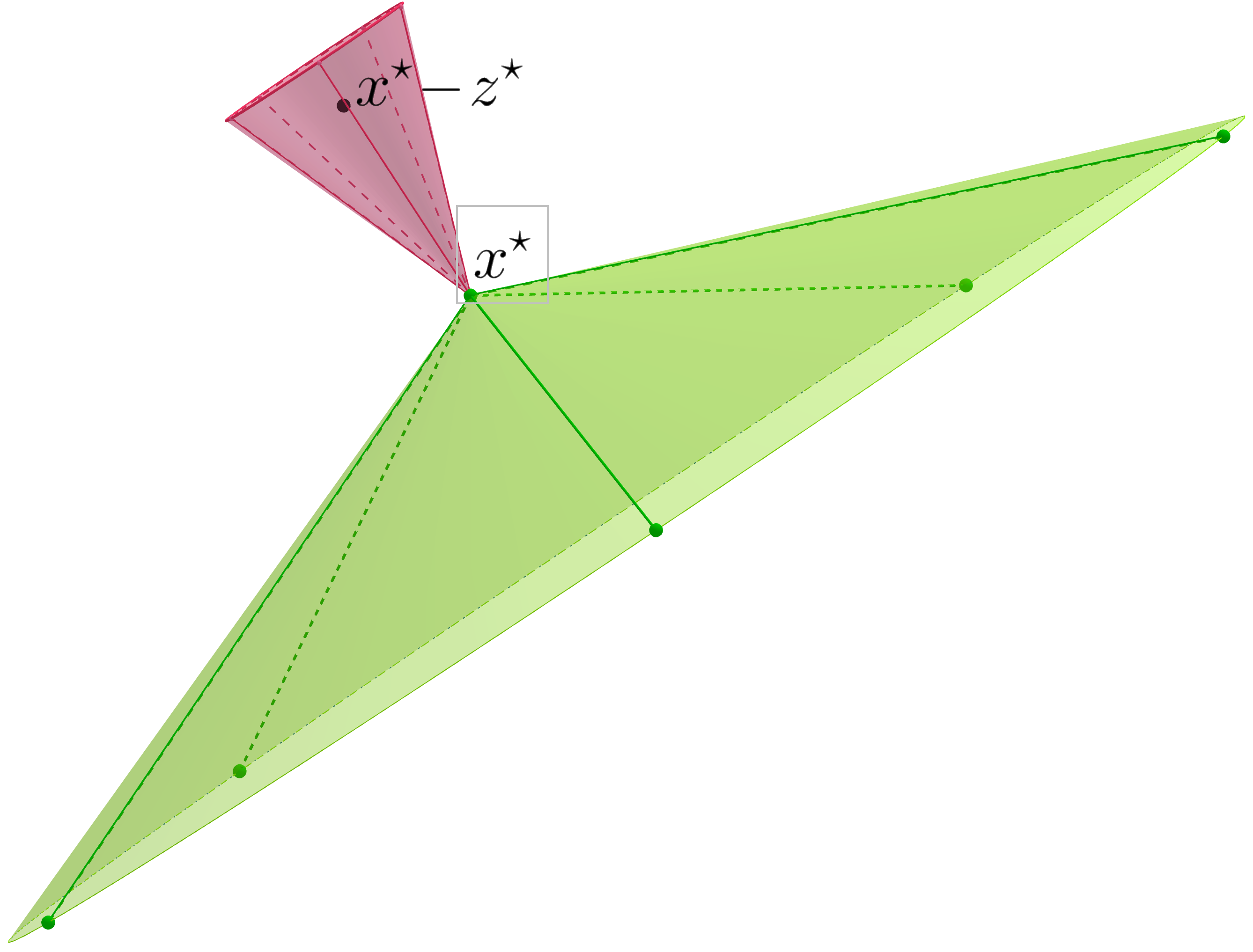}  \\
	\caption{Normal cone (red) at a point $x^*$ in the polytope $C$ (green).} 
	\label{fig:normalcone}
\end{figure}


\subsection{Local convergence of Douglas--Rachford splitting}

We start with the standard Douglas--Rachford splitting and then the damped iteration. Relation with some existing work in the literature is also discussed.

\subsubsection{The standard Douglas--Rachford splitting}

For standard Douglas--Rachford splitting method, for what follows we impose the global convergence as an assumption, \ie \iref{A:D} holds.

\begin{theorem}[Finite termination of DR]\label{thm:local_dr}
For the feasibility problem \eqref{eq:feasibility} and the Douglas--Rachford iteration \eqref{eq:dr}, suppose Assumptions \iref{A:C}-\iref{A:D} hold. 
Then $\seq{\uk,\xk,\zk}$ converges to $(\xsol,\xsol, \zsol)$ with $\zsol \in \fix(\fDR)$ being a fixed point and $\xsol = \proj_{S}(\zsol)$. 
If, moreover, the non-degeneracy condition \eqref{eq:ndc} holds, then $\seq{\uk,\xk,\zk}$ converges to $\pa{\xsol,\xsol,\zsol}$ finitely. 
\end{theorem}
%
%
\begin{remark}
It is worth noting that Theorem \ref{thm:local_dr} also holds true for the convex setting. 
In \cite{BauschkeFiniteDR15} the authors study DR for solving convex affine-polyhedral feasibility problem, and impose the following condition for finite convergence
\beq\label{eq:slater}
S \mcap \mathrm{int}\pa{C} \neq \emptyset ,
\eeq
which does not hold for the non-convex case as the interior of $C$ in \eqref{eq:feasibility} can be empty; See also Section \ref{sec:puzzles} the puzzles for which \eqref{eq:slater} fails. 
In \cite{bauschke2017finite}, when the non-convex set is finite, finite termination is proved given that the other set is an affine subspace or a half-space. 
In comparison, our result here does not need the set to be finite and provides an extension to that of \cite{BauschkeFiniteDR15}, as we characterize the situation where finite convergence happens but \eqref{eq:slater} fails. 
\end{remark}
\begin{proof}
The imposed global convergence of \eqref{eq:dr} means
\beq\label{eq:conv}
\zk \to \zsol \in \fix(\fDR) 
\qandq 
\uk, \xk \to \xsol = \proj_{S}(\zsol) . 
\eeq
The prox-regularity of $C$ at $\xsol$ for $\xsol-\zsol$ and the non-degeneracy condition \eqref{eq:ndc} imply that there exists an open set $\calB$ such that 
$$ 2\xsol-\zsol\in \calB\subset \normal{C}\pa{\xsol}+\xsol \qandq \proj_{C}\pa{\calB} = \{\xsol\}.$$
By the definition of convergence, there must therefore exist $K\in\bbN$ such that $ 2\xkp-\zkp \in \calB $ for all $k\geq K$. Consequently, by the update of $\ukp$ in \eqref{eq:ddr},
\[
\ukp = \proj_{C}(2\xkp - \zk) = \xsol
\]
which is the finite convergence of $\ukp$. 

For the update of $\xk$ in \eqref{eq:dr}, this time we have directly
\[
\begin{aligned}
\xkp - \xsol 
&= \proj_{S}(\zk - \zsol). 
\end{aligned}
\]
For $\zkp$, let $K > 0$ be such that $\uk = \xsol$ for all $k \geq K$, we have
\[
\begin{aligned}
\zkp - \zsol
= (\zk-\zsol) + (\ukp-\xsol) - (\xkp - \xsol)	
&= (\zk-\zsol) - (\xkp - \xsol) \\
&= \pa{ \Id - \proj_{S} } (\zk - \zsol)  \\
&= \pa{ \Id - \proj_{S} }^{k+1-K} (z_{K} - \zsol)  . 
\end{aligned}
\]
Since $\zk \to \zsol$ and $\pa{ \Id - \proj_{S} }^{k+1-K} = \Id - \proj_{S}$, we have
\[
\begin{aligned}
0 
= \lim_{k\to\pinf} \zkp - \zsol 
= \lim_{k\to\pinf} \pa{ \Id - \proj_{S} }^{k+1-K} (z_{K} - \zsol)  
&= \pa{ \Id - \proj_{S} } (z_{K} - \zsol) 
= \zkp - \zsol  ,
\end{aligned}
\]
which means $\zk = \zsol$ for all $k > K$, hence finite termination of $\zk$. 
The finite convergence of $\xk$ follows naturally that of $\zk$, and we conclude the proof. 
\end{proof}


\paragraph{Different order of update}
In \eqref{eq:dr}, the order of the projection operators can be switched which results in the following iteration
\beq\label{eq:dr_diff}
\begin{aligned}
	\xkp &\in \proj_{C} (\zk)  , \\
	\ukp &= \proj_{S}\pa{2\xkp - \zk} , \\
	\zkp &= \zk + \ukp - \xkp  .
\end{aligned}
\eeq
The corollary below shows that the finite termination holds for \eqref{eq:dr_diff}. 

\begin{corollary}\label{cor:local_dr}
For the feasibility problem \eqref{eq:feasibility} and the Douglas--Rachford iteration \eqref{eq:dr_diff}, suppose Assumptions \iref{A:C}-\iref{A:D} hold. 
Then $\seq{\uk,\xk,\zk}$ converges to $(\xsol,\xsol, \zsol)$ with $\zsol \in \fix(\fDR)$ being a fixed point and $\xsol \in \proj_{C}(\zsol)$. 
If, moreover, $C$ is prox-regular at $\xsol$ for $\zsol-\xsol$ and the following non-degeneracy condition holds, 
\beq\label{eq:ndc_diff}
-(\xsol-\zsol) \in \mathrm{int}\Pa{ \normal{C}\pa{\xsol} } ,
\eeq
then $\seq{\uk,\xk,\zk}$ converges to $\pa{\xsol,\xsol,\zsol}$ in a finite number of iterations.
\end{corollary}
\begin{proof}
Following the argument of the proof of Theorem \ref{thm:local_dr}, we can easily derive the finite termination of $\xk$ under the new non-degeneracy condition \eqref{eq:ndc_diff}. 
In turn, for $k$ large enough, we have for $\ukp$ that
\[
\begin{aligned}
\ukp - \xsol
&= \proj_{S}\pa{2\xkp - \zk} - \proj_{S}\pa{2\xsol - \zsol}  \\
&= 2\proj_{S}\pa{\xkp - \xsol} - \proj_{S}\pa{\zk - \zsol} \\
&= - \proj_{S}\pa{\zk - \zsol}  .
\end{aligned}
\]
As a result for $\zk$,
\[
\begin{aligned}
\zkp - \zsol
= (\zk-\zsol) + (\ukp-\xsol) - (\xkp - \xsol) 
&= (\zk-\zsol) + (\ukp-\xsol)  \\
&= \pa{\Id - \proj_{S}} (\zk-\zsol) ,
\end{aligned}
\]
which is the same as the last part of proof of Theorem \ref{thm:local_dr}, hence we conclude the proof. 
\end{proof}

\subsubsection{The damped Douglas--Rachford splitting}

We now turn to the local convergence analysis of the damped Douglas--Rachford splitting \eqref{eq:ddr}, for which we have the following result.

\begin{theorem}[Local convergence of dDR]\label{thm:local_ddr}
For the feasibility problem \eqref{eq:feasibility} and the damped Douglas--Rachford iteration \eqref{eq:ddr}, suppose that Assumptions \iref{A:C}-\iref{A:intersection} hold and \eqref{eq:ddr} is ran under the conditions of Theorem~\ref{thm:global_convergence}, then $\pa{\uk,\xk,\zk} \to (\xsol,\xsol, \zsol)$ with $\zsol \in \fix(\fdDR)$ being a fixed point and $\xsol$ a stationary point of \eqref{eq:dist_C}. 
If, moreover, condition \eqref{eq:ndc} holds, then 
\begin{enumerate}[label={\rm (\roman{*})}]
\item $\uk$ converges in finite number of iterations. 
$\uk = \xsol $.

\item Let $\eta = \frac{\gamma}{1+\gamma}$, it holds
$\norm{\zk - \zsol} = O(\eta^k) $.
\end{enumerate}
\end{theorem}
\begin{proof}
The finite convergence of $\uk$ follows the argument of the proof of Theorem \ref{thm:local_dr}. 
For the update of $\xk$ in \eqref{eq:ddr}, since $S$ is an affine subspace, $\proj_{S}$ is~linear
\[
\begin{aligned}
\xkp - \xsol 
&= \sfrac{1}{1+\gamma} \Pa{ \zk + \gamma \proj_{S}(\zk) } - \xsol  \\
&= \sfrac{1}{1+\gamma} \Pa{ \zk + \gamma \proj_{S}(\zk) } - \sfrac{1}{1+\gamma} \Pa{ \xsol + \gamma \proj_{S}(\zsol) } \\
&= \sfrac{1}{1+\gamma} (\zk - \xsol) + \sfrac{\gamma}{1+\gamma} \proj_{S}(\zk - \zsol). 
\end{aligned}
\]
Now for $\zkp$, let $K > 0$ be such that $\uk = \xsol$ for all $k \geq K$, we have
\[
\begin{aligned}
\zkp - \zsol
&= (\zk-\zsol) + (\ukp-\usol) - (\xkp - \xsol)	\\ 
&= (\zk-\zsol) - (\xkp - \xsol) \\
&= (\zk-\zsol) - \sfrac{1}{1+\gamma} (\zk - \xsol) - \sfrac{\gamma}{1+\gamma} \proj_{S}(\zk - \zsol) \\
&= \sfrac{\gamma}{1+\gamma} \pa{ \Id -  \proj_{S} } (\zk - \zsol)  .
\end{aligned}
\]
Note that the spectral radius of the matrix appears above is
\[
\rho\Pa{ \tfrac{\gamma}{1+\gamma}(\Id-\proj_{S}) } = \sfrac{\gamma}{1+\gamma}  .
\]
Combined with the fact the matrix is symmetric and normal, owing to Lemma \ref{lem:rate-M} we conclude $\frac{\gamma}{1+\gamma}$ is the local linear convergence rate of $\norm{\zk-\zsol}$. 
\end{proof}

\begin{remark}
In \cite{li2016douglas}, the authors also discuss the local linear convergence of damped DR under the following constraint qualification condition
\beq\label{eq:cq}
\normal{S}\Pa{\proj_{S}(\xsol)} \mcap - \normal{C}(\xsol) = 0 .
\eeq
As shown in \cite[Proposition 2]{li2016douglas}, such a condition allows to show $\xsol \in C \cap S$ and $\zsol = \xsol$; See Example \ref{exp:circle_line} which satisfies the above condition. The update of $\xkp$ in \eqref{eq:ddr} yields
\[
\begin{aligned}
\sfrac{1+\gamma}{\gamma} \pa{ \zsol - \xsol } = {\zsol -  \proj_{S} (\zsol)} &\in \normal{S}\Pa{\proj_{S}(\xsol)}  , \\
\zsol - \xsol &\in - \normal{C}\pa{\xsol} .
\end{aligned}
\]
This implies that only the fixed-points $\zsol$ such that $\zsol = \xsol$ satisfy the qualification condition \eqref{eq:cq}. 
In comparison, our non-degeneracy condition is more general than \eqref{eq:cq} in the sense that we only focus on $\normal{C}\pa{\xsol}$ and does not need the intersection of $\normal{S}\pa{\proj_{S}(\xsol)} \mcap - \normal{C}(\xsol)$ to be $0$, and our result holds for all fixed-points of $\fix(\fdDR)$. 
\end{remark}

\begin{remark}
When $S$, instead of being an affine subspace, has locally smooth curvature around $\xsol$, then according to the result of \cite{liang2017local}, one can show that for any $\eta \in ]\frac{\gamma}{1+\gamma}, 1[$ there holds $\norm{\zk-\zsol} = O(\eta^k)$.
\end{remark}

\subsection{Sudoku and $s$-queens puzzles}\label{sec:puzzles}

In this part, we specialize the above result to Sudoku and $s$-queens puzzles. Examples of these two puzzles are provided in Figure \ref{fig:puzzles} below. 
%

\begin{figure}[!ht]
	\centering
	\subfloat[Sudoku puzzle]{ \includegraphics[width=0.35\linewidth]{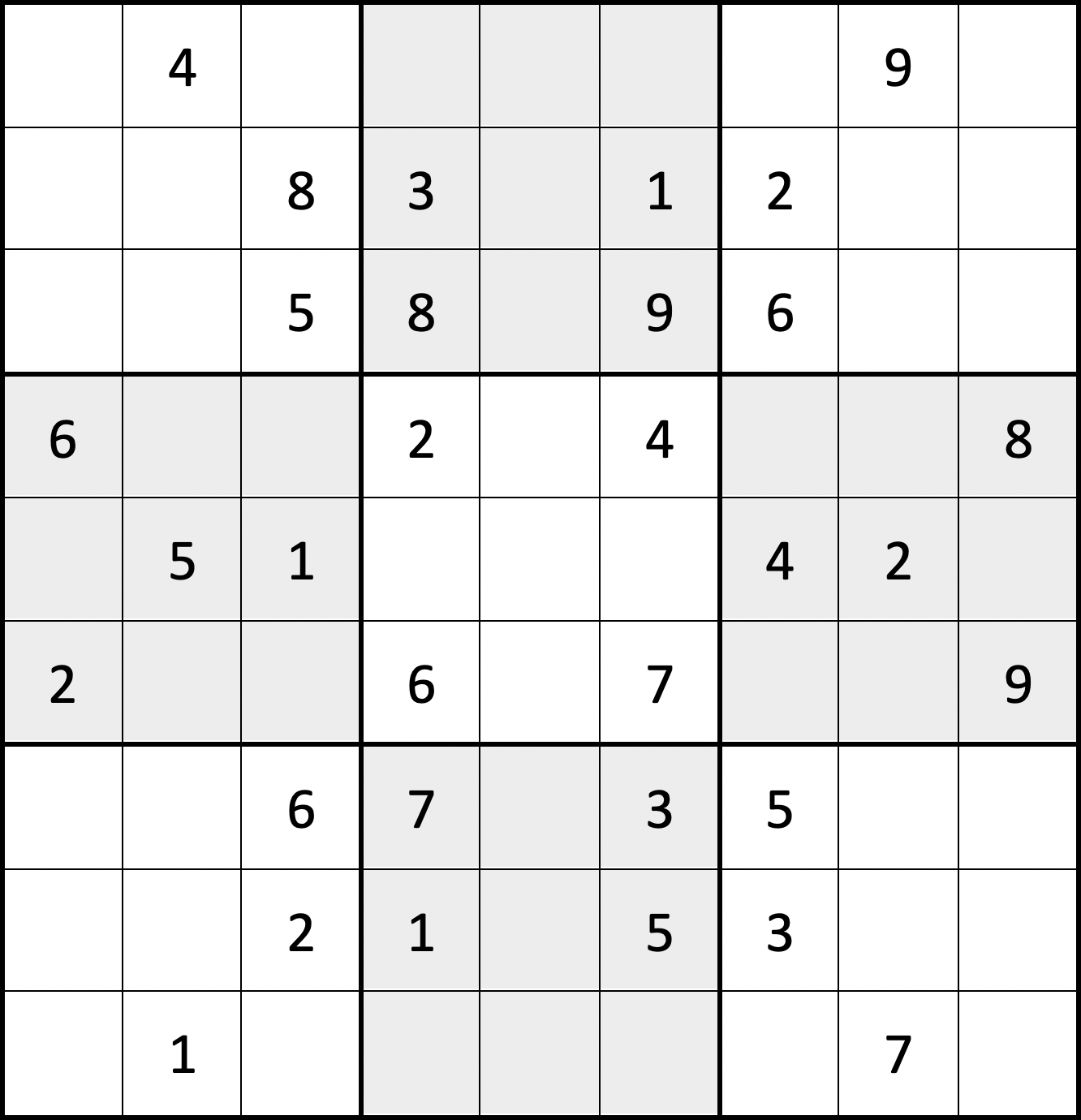} } \hspace{6pt}
	\subfloat[Eight queens puzzle]{ \includegraphics[width=0.362\linewidth]{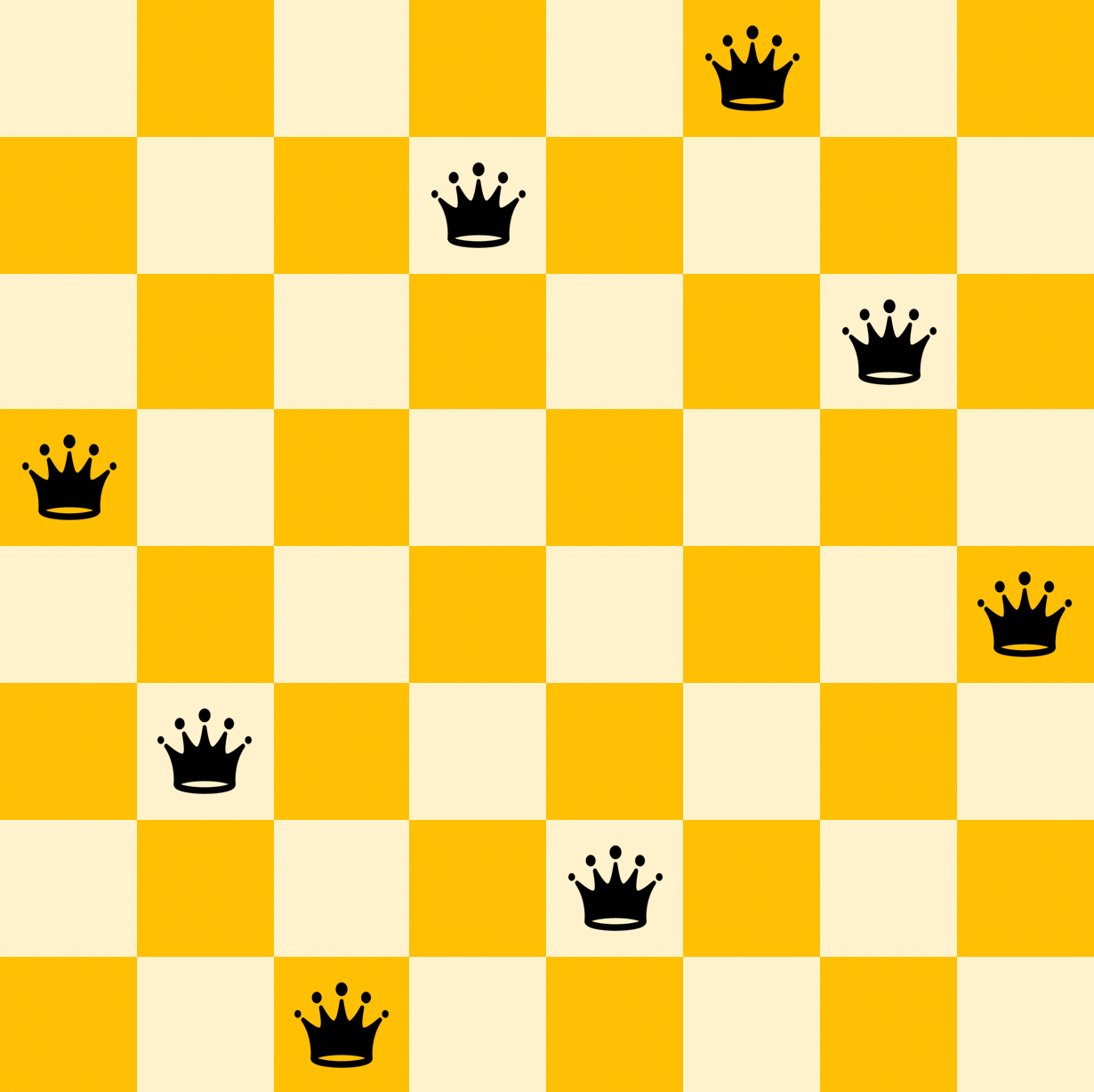} } \\
	\caption{Examples of Sudoku and eight queens. The goal of Sudoku is to complete the grid such that each row, column, and $3\times3$ square contains all the digits from $1$ to $9$. The goal of eight queens is to place eight chess queens on an $8\times8$ board such that no two queens share the same row/column/diagonal.} 
	\label{fig:puzzles}
\end{figure}


\subsubsection{Sudoku puzzle}

A standard Sudoku puzzle is shown in Figure \ref{fig:puzzles} (a), which we generalize to grids of size $s\times s$ with the basic setting and rules: 
\begin{itemize}
\item A partially complete $s\times s$ grid is provided. 
\item Each {\it column}, each {\it row} and each of the {\it $s$ sub-grids of size $\sqrt{s}\times\sqrt{s}$} that compose the grid contain all of the digits from $1$ to $s$. 
\end{itemize}
Based on the rules, we can easily formulate the Sudoku puzzle as feasibility problem. Here we consider the formulation proposed in \cite{schaad2010modeling}, which formulates Sudoku as binary feasibility problem. We also refer to \cite{artacho2014recent} for studies on Sudoku puzzle and Douglas--Rachford splitting method. 

Each digit from $1$ to $s$ is lifted to the set $[0,1]^s$, making the full puzzle an $s\times s\times s$ binary cube. Figure~\ref{fig:lift} (a) shows a feasible row of the lifted problem represented as a binary $s\times s$ square. Equivalently, we can say that any digit from $1$ to $s$ is a permutation of unit vector $e = \ba{1,0,\ldots,0}$. This leads to four Sudoku feasibility constraints:
\begin{itemize}
	\item Each {\it row} of the cube, \ie $C_1(:, j,k),\ j,k \in \ba{1,\ldots,s}$, is the permutation of $e$; See Figure \ref{fig:lift} (b).
	\item Each {\it column} of the cube, \ie$C_2(i,:,k),\ i,k \in \ba{1,\ldots,s}$, is the permutation of $e$; See Figure \ref{fig:lift} (c).
	\item Each {\it pillar} of the cube, \ie$C_3(i, j, :),\ i,j \in \ba{1,\ldots,s}$, is the permutation of $e$; See Figure \ref{fig:lift} (d).
	\item For each $k \in \ba{1,\ldots,s}$, each of the {\it $s$ sub-grids} is the permutation of $e$, \ie $C_4(\sqrt{s}(i-1)+1:\sqrt{s}i, \sqrt{s}(j-1)+1:\sqrt{s}j, k),\ i,j \in \ba{1,\ldots,\sqrt{s}}$; See Figure \ref{fig:lift} (e).
\end{itemize}
The partially completed grid forms the last constraint set
\begin{itemize}
\item $C_5$ is the constraint of the provided numbers.
\end{itemize}

\begin{figure}[!ht]
	\centering
	\subfloat[Lifted row]{ \includegraphics[width=0.19\linewidth]{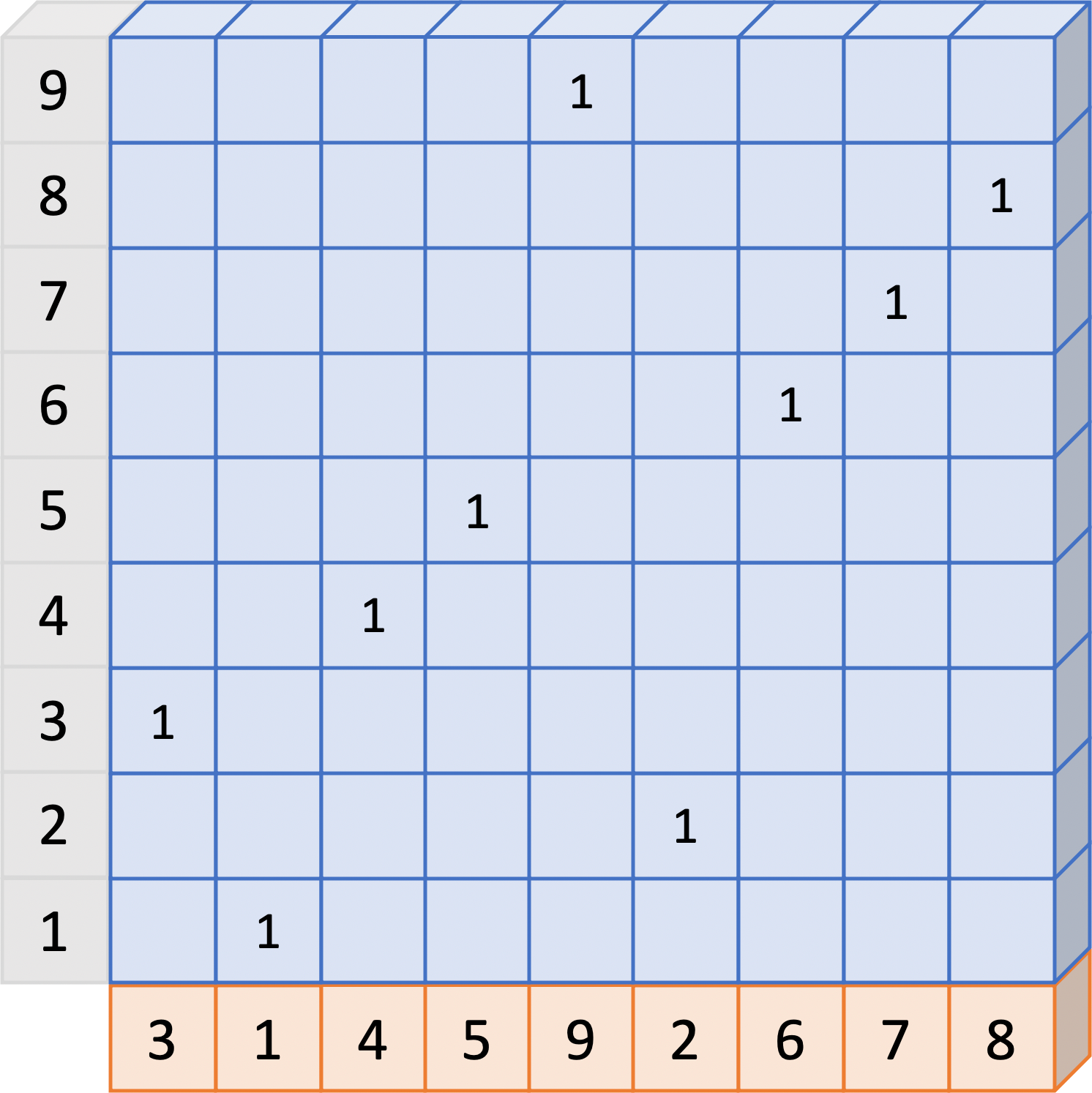} } \hspace{-8pt}
	\subfloat[$C_1$]{ \includegraphics[width=0.19\linewidth]{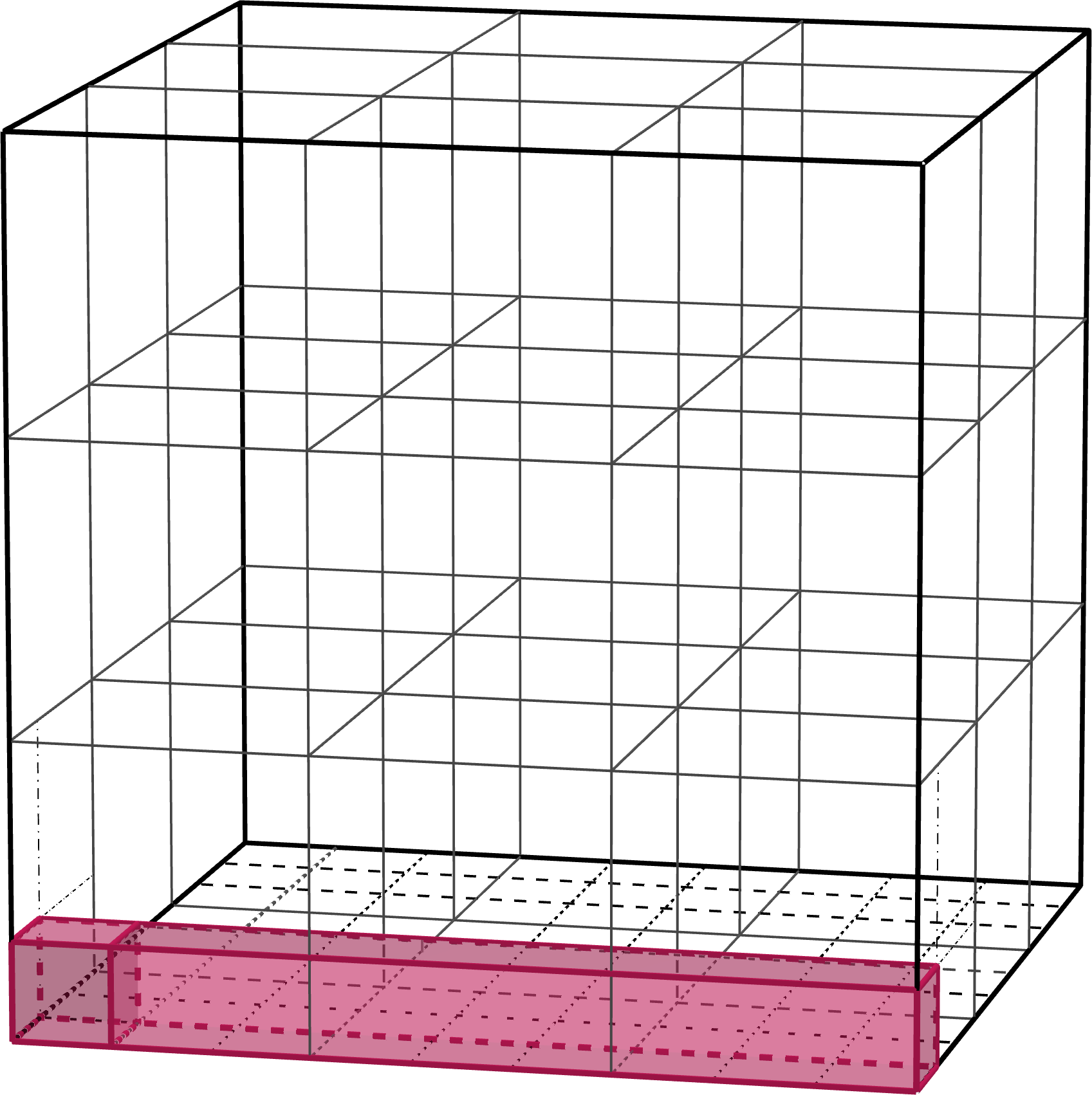} } \hspace{-8pt}
	\subfloat[$C_2$]{ \includegraphics[width=0.19\linewidth]{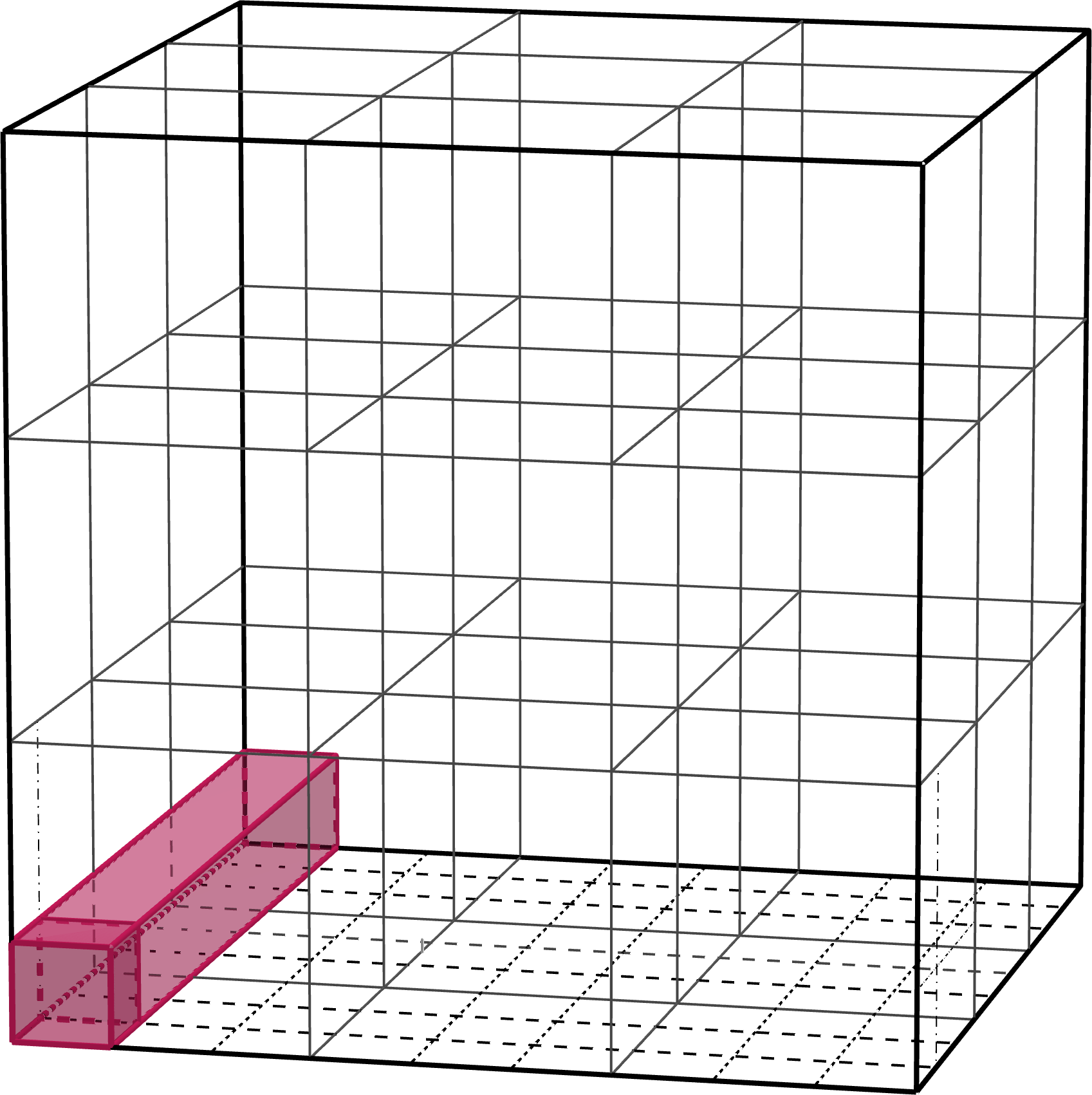} } \hspace{-8pt}
	\subfloat[$C_3$]{ \includegraphics[width=0.19\linewidth]{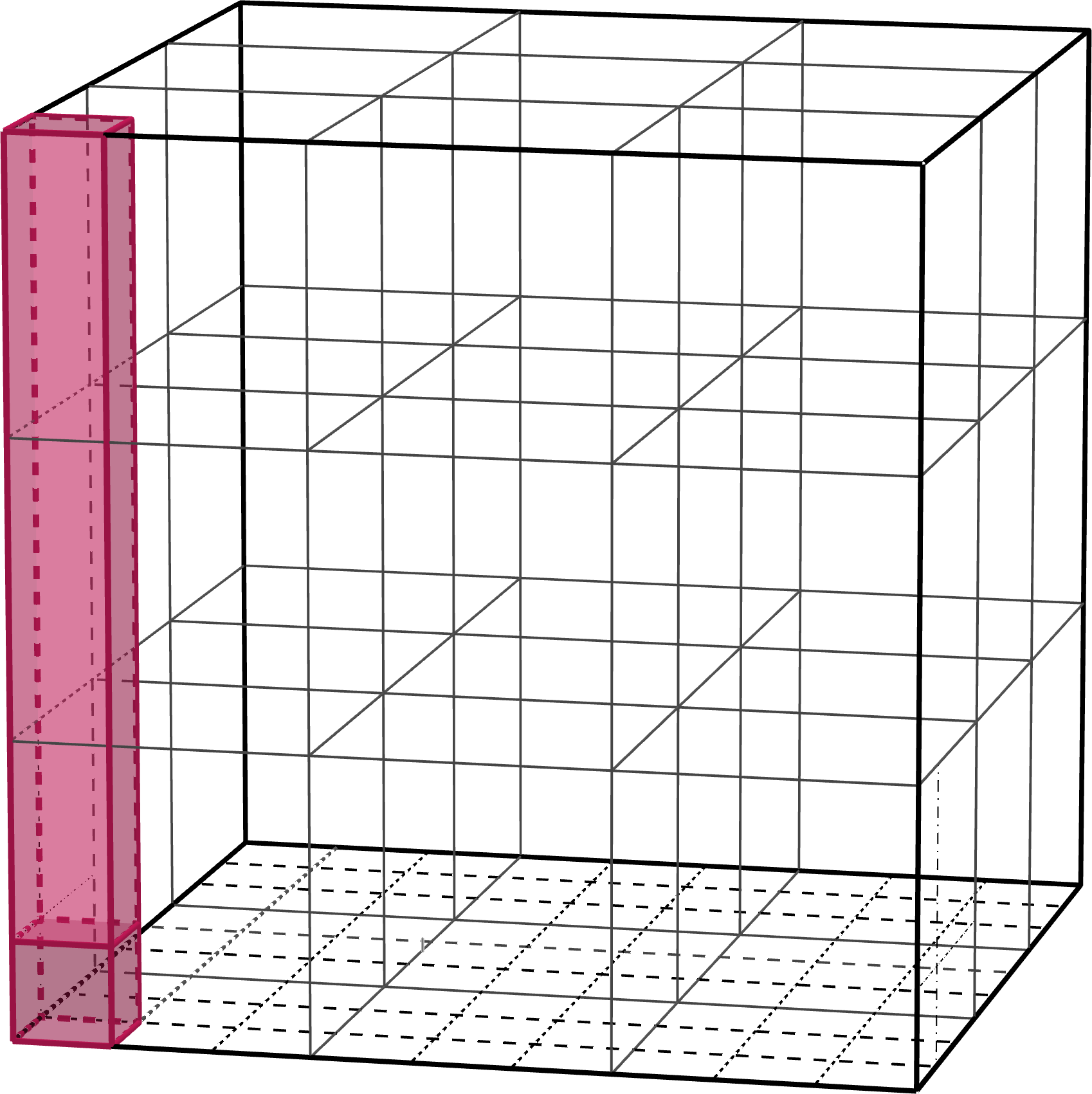} } \hspace{-8pt}
	\subfloat[$C_4$]{ \includegraphics[width=0.19\linewidth]{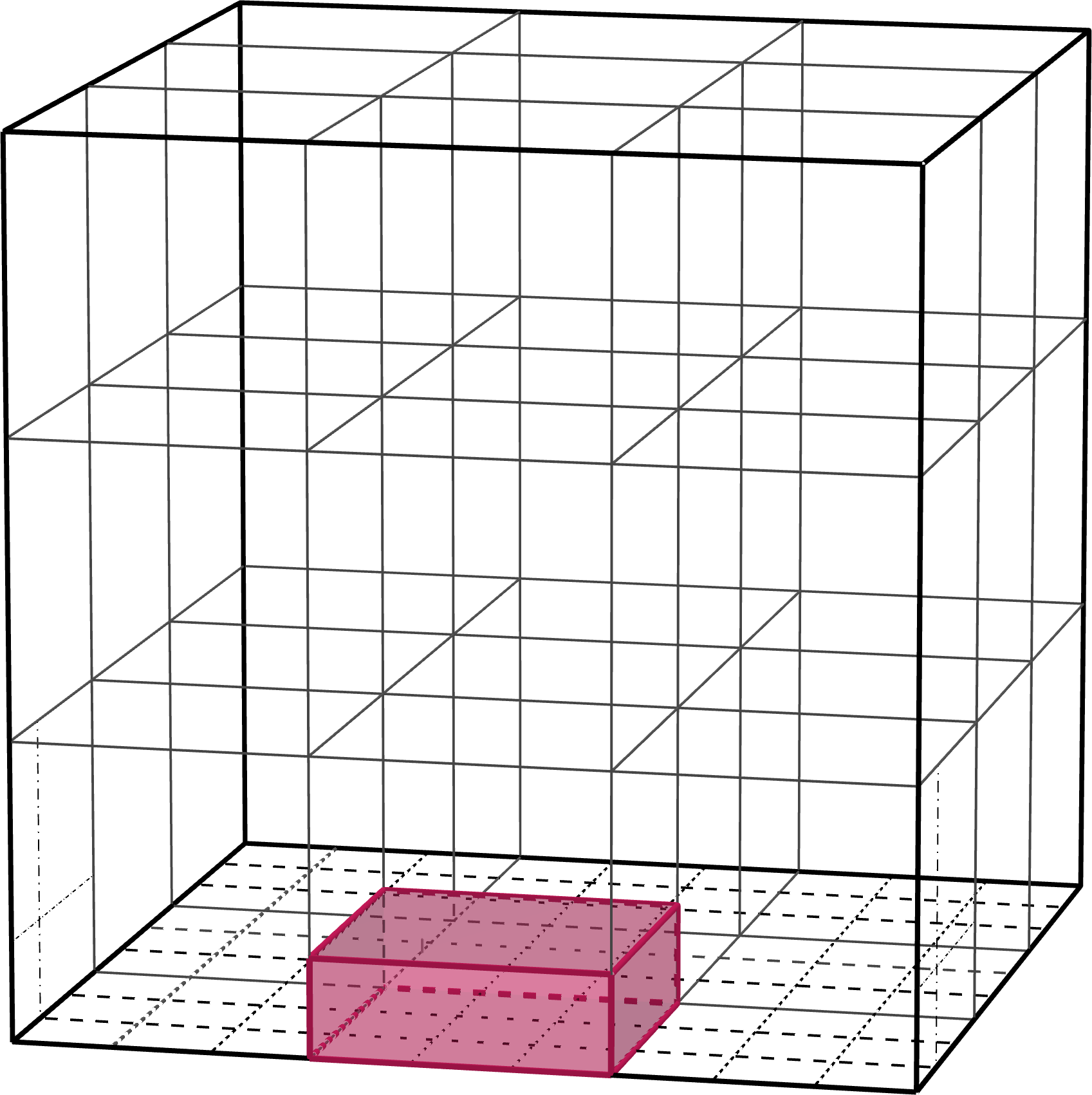} } \\
	\caption{Lifted Sudoku problem. (a) the lifted representation of a row of numbers. (b)-(e) show what is a {\it lifted row/column/pillar/sub-grid} respectively.} 
	\label{fig:lift}
\end{figure}

At this point, solving the Sudoku puzzle is equivalent to solve the following feasibility problem of the five constraint sets
\beq\label{eq:sudoku_5set}
\find ~~x~ \in \bbR^{s\times s\times s} \enskip \st \enskip  x \in C_1 \cap C_2 \cap C_3 \cap C_4 \cap C_5 .
\eeq
To obtain the product space formulation, let $\bcH=\underset{\textrm{$5$ times}}{\underbrace{\bbR^{s\times s\times s} \times \cdots \times \bbR^{s\times s\times s}}}$, $\bcC \eqdef C_{1} \times \dotsm \times C_{5}$, and $\bcS = \ba{\bmx=(x_i)_i\in\bcH : x_1=\dotsm=x_5}$.


\begin{proposition}[Local convergence of DR]\label{prop:sudoku_dr}
For the Sudoku puzzle \eqref{eq:sudoku_5set} and Douglas--Rachford splitting \eqref{eq:dr_msets}, suppose Assumptions \iref{A:C}-\iref{A:D} hold. 
Then $\seq{\bmu_{k},\xk,\bmz_{k}}$ converges to $(\bcK(\xsol),\xsol, \bmz^\star)$ with $\bmz^\star \in \fix(\fDR)$ being a fixed point and $\xsol = \frac{1}{5}\sum_{i=1}^{5} z_i^\star$. 
If, moreover, for $i=1,\ldots,4$, $C_i$ is prox-regular at $\xsol$ for $\xsol - z_{i}^\star$ and the following non-degeneracy condition holds 
\beq\label{eq:ndc_Ci}
\xsol - z_i^\star \in \mathrm{int}\Pa{ \normal{C_i}\pa{\xsol} } .
\eeq
Then for all $k$ large enough, there holds
\begin{itemize}
\item 
$u_{i,k} = \xsol$ for $i=1,\ldots,4$, 
\item 
$\norm{\bmz_k-\bmz^\star} = O(\eta^k)$ with $\eta = \frac{\sqrt{5}}{5}$. 
\end{itemize}
\end{proposition}
\begin{proof}
Denote $\bmx_{k+1} = \bcK(\xkp)$, from the updates of $\bmx_k$, we have that
\[
\bmx_{k+1} - \bmx^\star = \proj_{\bcS}(\bmz_{k} - \bmz^\star)
\]
with $\proj_{\bcS}
= \frac{1}{5} \mathbf{1}_{5\times5} \otimes \Id_{s^3\times s^3}   $, 
where $\mathbf{1}_{5\times5}$ stands for matrix of all $1$ and $\otimes$ for Kronecker product. 

The separability of $\proj_{\bcC}$ and the definition of projection operator lead to, for each $i=1,\ldots,5$
\[
\begin{aligned}
u_{i,k+1} &= \proj_{C_i} (2x_{k+1}-z_{i,k})  \\
\qandq
u_{i,k+1} - u_{i}^\star &= \proj_{C_i} (2x_{k+1}-z_{i,k}) - \proj_{C_i} (2x^\star-z_{i}^\star) .
\end{aligned}  
\]
Under the non-degeneracy condition \eqref{eq:ndc_Ci}, apply the argument of Theorem \ref{thm:local_dr} to obtain the finite convergence of $u_{i,k}$ for $i=1,\ldots,4$.  
For $C_5$, since its projection operator is linear, we have
\[
\begin{aligned}
u_{5,k+1} - u_{5}^\star 
&= \proj_{C_5} (2x_{k+1}-z_{5,k}) - \proj_{C_5} (2x^\star-z_{5}^\star) \\
&= 2\proj_{C_5} (x_{k+1}-x^\star) - \proj_{C_5} (z_{5,k}-z_{5}^\star) . 
\end{aligned}
\]
As a result, for $k$ large enough there holds
\[
\begin{aligned}
\bmu_{k+1} - \bmu^\star
&= 
2 \begin{bmatrix} \mathbf{0}_{4s^3\times4s^3} & \\ & \proj_{C_5} \end{bmatrix}
\proj_{\bcS}(\bmz_{k} - \bmz^\star) - \begin{bmatrix} \mathbf{0}_{4s^3\times4s^3} & \\ & \proj_{C_5} \end{bmatrix} (\bmz_{k} - \bmz^\star) . 
\end{aligned}
\]
Let $\proj_{\bcC} \eqdef \begin{bmatrix} \mathbf{0}_{4s^3\times4s^3} & \\ & \proj_{C_5} \end{bmatrix}$ and back to $\bmz_{k+1} - \bmz^\star$, we get
\[
\begin{aligned}
\bmz_{k+1} - \bmz^\star
&= (\bmz_{k} - \bmz^\star) + (\bmu_{k+1} - \bmu^\star) - (\bmx_{k+1} - \bmx^\star) \\
&= \Pa{ \bId + 2\proj_{\bcC} \proj_{\bcS} - \proj_{\bcC} - \proj_{\bcS} } (\bmz_{k} - \bmz^\star) ,
\end{aligned}
\] 
Since $\proj_{C_5}$ is the projection operator onto a subspace, so is $\proj_{\bcC}$. 
As a result, the linear convergence rate is the cosine of the Friedrichs angle $\theta_F$ between the subspace of $\proj_{\bcC}$ and that of $\proj_{\bcS}$. 
%
We now need to analyze the singular values of $\proj_{\bcC}\proj_{\bcS}$, which essentially is the SVD of $\proj_{C_5}\proj_{S}$ where
\[
\proj_{S}
= \sfrac{1}{5} \mathbf{1}_{1\times5} \otimes \Id_{s^3\times s^3}   .
\]
We have
\begin{itemize}
\item $\proj_{C_5}$ is diagonal matrix with only $0$ and $1$.
\item $\proj_{S}$ has a unique singular value which is $\frac{\sqrt{5}}{5}$.
\end{itemize}
As a result, $\proj_{C_5}\proj_{S}$ has only two singular values which are $0$ and $\frac{\sqrt{5}}{5}$. Hence we conclude the proof. 
\end{proof}

Next we present result for the damped Douglas--Rachford splitting \eqref{eq:ddr}.

\begin{proposition}[Local convergence of dDR]\label{prop:sudoku_ddr}
For the Sudoku puzzle \eqref{eq:sudoku_5set} and the damped Douglas--Rachford splitting \eqref{eq:ddr_msets}, suppose Assumptions \iref{A:C}-\iref{A:intersection} hold and \eqref{eq:ddr_msets} is ran under the conditions of Theorem~\ref{thm:global_convergence}, then $\pa{\bmu_k,\bmx_k, \bmz_k}$ converges to $(\bmx^\star,\bmx^\star, \bmz^\star)$ with $\bmz^\star$ being a fixed point and $\bmx^\star$ a stationary point of $\min_{\bmx} \left\{\dist^2(\bmx, \bcS)~~\mathrm{s.t.}~~\bmx\in\bcC\right\}$.
If, moreover, the non-degeneracy condition \eqref{eq:ndc_Ci} holds for $C_{1,\ldots,4}$, then for all $k$ large enough, it holds
\begin{itemize}
\item 
$u_{i,k} = \xsol$ for $i=1,\ldots,4$, 
\item 
$\norm{\bmz_k-\bmz^\star} = O(\eta^k)$ with $\eta = \frac{ {2\gamma + 5} + \sqrt{25 - 16\gamma^2} }{ 10(1+\gamma) }$. 
\end{itemize}
\end{proposition}
\begin{proof}
From the updates of $\bmx_{k+1}$, we have that
\[
\begin{aligned}
\bmx_{k+1} - \bmx^\star 
&= \sfrac{1}{1+\gamma}\Pa{ \bmz_{k} + \gamma \proj_{\bcS}(\bmz_{k}) } - \sfrac{1}{1+\gamma}\Pa{ \bmz^\star + \gamma \proj_{bcS}(\bmz^\star) } \\
&= \sfrac{1}{1+\gamma} \pa{\bmz_{k}-\bmz^\star} +  \sfrac{\gamma}{1+\gamma}  \proj_{\bcS}(\bmz_{k} - \bmz^\star)
\end{aligned}
\]
with $\proj_{\bcS} = \sfrac{1}{5} \mathbf{1}_{5\times5} \otimes \Id_{s^3\times s^3}   $. 
For $\bmu_{k}$, the finite termination of $u_{i,k},~i=1,\ldots,4$ follows from the proof of Proposition \ref{prop:sudoku_dr}. 
For $C_5$, again we have
\[
u_{5,k+1} - u_{5}^\star = 2\proj_{C_5} (x_{5,k+1}-x_{5}^\star) - \proj_{C_5} (z_{5,k}-z_{5}^\star)  . 
\]
Let $\proj_{\bcC} \eqdef \begin{bmatrix} \mathbf{0}_{4s^3\times4s^3} & \\ & \proj_{C_5} \end{bmatrix}$, then %
\[
\begin{aligned}
\bmu_{k+1} - \bmu^\star
&= 2\proj_{\bcC} (\bmx_{k+1} - \bmx^\star) - \proj_{\bcC} (\bmz_{k} - \bmz^\star) \\ 
&= \sfrac{2}{1+\gamma} \proj_{\bcC} \pa{\bmz_{k}-\bmz^\star} +  \sfrac{2\gamma}{1+\gamma}  \proj_{\bcC} \proj_{\bcS}(\bmz_{k} - \bmz^\star) - \proj_{\bcC} (\bmz_{k} - \bmz^\star)\\
&= \sfrac{1-\gamma}{1+\gamma} \proj_{\bcC} \pa{\bmz_{k}-\bmz^\star} +  \sfrac{2\gamma}{1+\gamma}  \proj_{\bcC} \proj_{\bcS}(\bmz_{k} - \bmz^\star)  .
\end{aligned}
\]
Back to $\bmz_{k+1} - \bmz^\star$, we get
\[
\begin{aligned}
\bmz_{k+1} - \bmz^\star
&= (\bmz_{k} - \bmz^\star) + (\bmu_{k+1} - \bmu^\star) - (\bmx_{k+1} - \bmx^\star) \\
&= (\bmz_{k} - \bmz^\star) + \sfrac{1-\gamma}{1+\gamma} \proj_{\bcC} \pa{\bmz_{k}-\bmz^\star} +  \sfrac{2\gamma}{1+\gamma}  \proj_{\bcC} \proj_{\bcS}(\bmz_{k} - \bmz^\star) \\&\qquad  - \sfrac{1}{1+\gamma} \pa{\bmz_{k}-\bmz^\star} -  \sfrac{\gamma}{1+\gamma}  \proj_{\bcS}(\bmz_{k} - \bmz^\star) \\
&= \sfrac{1}{1+\gamma} \Pa{ {\gamma} \bId + {2\gamma} \proj_{\bcC} \proj_{\bcS} + \pa{1-\gamma} \proj_{\bcC} - {\gamma}  \proj_{\bcS} } (\bmz_{k} - \bmz^\star)  \\
&= \sfrac{1}{1+\gamma} \Pa{ \gamma \pa{ \bId + 2 \proj_{\bcC} \proj_{\bcS} - \proj_{\bcC} - \proj_{\bcS} } + \proj_{\bcC} } (\bmz_{k} - \bmz^\star)  .
\end{aligned}
\]
Denote $M_{\gamma} = \frac{1}{1+\gamma} \pa{ \gamma \pa{ \bId + 2 \proj_{\bcC} \proj_{\bcS} - \proj_{\bcC} - \proj_{\bcS} } + \proj_{\bcC} }$. 
Let $p,q$ be the rank of $\proj_{\bcC}$ and $\proj_{\bcS}$ respectively, also assume $p\leq q$ ({For the case $p\geq q$, similar result can be obtained}). Based on \cite{Bauschke2016}, there exists an orthogonal matrix $U$ such that
\[
\proj_{\bcC} =
U
\left[
\begin{array}{cc|cc}
\Id_{p} & 0 & 0 & 0\\
0  &  0_{p}  & 0 & 0\\
\hline
0 & 0 & 0_{q-p} & 0\\
0 & 0 & 0 & 0_{n-p-q}
\end{array}
\right]
U^*  \\
\qandq
\proj_{\bcS} =
U
\left[
\begin{array}{cc|cc}
\alpha^2 & \alpha \beta & 0 & 0\\
\alpha \beta  &  \beta^2  & 0 & 0\\
\hline
0 & 0 & \Id_{q-p} & 0\\
0 & 0 & 0 & 0_{n-p-q}
\end{array}
\right]
U^*  ,
\]
where $\alpha = \diag(\cos(\theta_1),\ldots,\cos(\theta_{p}))$ and $\beta = \diag(\sin(\theta_1),\ldots,\sin(\theta_{p}))$ with $\theta_{i, i=1,..p}$ being the principal angles between the subspaces of $\proj_{\bcC}$ and $\proj_{\bcS}$. 
Consequently,
\[
\bId + 2 \proj_{\bcC} \proj_{\bcS} - \proj_{\bcC} - \proj_{\bcS}
=
U
\left[
\begin{array}{cc|cc}
\alpha^2 & \alpha \beta & 0 & 0\\
- \alpha \beta  &  \alpha^2  & 0 & 0\\
\hline
0 & 0 & 0_{q-p} & 0\\
0 & 0 & 0 & \Id_{n-p-q}
\end{array}
\right]
U^*   .
\]
Therefore, we have
\[
M_{\gamma}
= 
\sfrac{1}{1+\gamma}
U
\left[
\begin{array}{cc|cc}
\gamma \alpha^2 + \Id_{p} & \gamma \alpha \beta & 0 & 0\\
- \gamma \alpha \beta  &  \gamma \alpha^2  & 0 & 0\\
\hline
0 & 0 & 0_{q-p} & 0\\
0 & 0 & 0 & \gamma \Id_{n-p-q}
\end{array}
\right]
U^*
\]
Clearly, $0$ and $\frac{\gamma}{1+\gamma}$ are two eigenvalues of the matrix. For the top left block of the above matrix, as it is block diagonal, we have the following characteristic polynomial
\[
0 = \prod_{i=1}^{p} \bPa{ \Pa{ \tfrac{\gamma \alpha_i^2 + 1}{1+\gamma} - \lambda }\Pa{ \tfrac{\gamma \alpha_i^2}{1+\gamma} - \lambda } + \tfrac{\gamma^2 \alpha_i^2\beta_i^2}{(1+\gamma)^2} } .
\]
Solving the quadratic equation for each $i$ we get
\[
	\lambda_i = \qfrac{ {2\gamma \alpha_i^2 + 1} \pm \sqrt{1 - 4\gamma^2 \alpha_i^2 \beta_i^2} }{ 2(1+\gamma) }  .
\]
As in the proof of Proposition \ref{prop:sudoku_dr}, we have that $\alpha_i = \frac{\sqrt{5}}{5}$ for all $i=1,\ldots,p$, therefore  $M_{\gamma}$ has only $4$ distinct eigenvalues which are
\[
0, \quad 
\sfrac{ {2\gamma + 5} - \sqrt{25 - 16\gamma^2} }{ 10(1+\gamma) } , \quad
\sfrac{\gamma}{1+\gamma} \qandq
\sfrac{ {2\gamma + 5} + \sqrt{25 - 16\gamma^2} }{ 10(1+\gamma) }  .
\]
We also have
\[
\begin{aligned}
\sfrac{\gamma}{1+\gamma} \leq \sfrac{ {2\gamma + 5} + \sqrt{25 - 16\gamma^2} }{ 10(1+\gamma) } &: \gamma \in ]0, 1] , \\
\sfrac{\gamma}{1+\gamma} \geq \abs{ \sfrac{ {2\gamma + 5} + \sqrt{25 - 16\gamma^2} }{ 10(1+\gamma) } } &: \gamma \in ]1, +\infty[  \\
\end{aligned}
\]
Next we focus on $\gamma \in ]0, 1]$ and show $\eta = \frac{ {2\gamma + 5} + \sqrt{25 - 16\gamma^2} }{ 10(1+\gamma) }$ is the convergence rate, to this end, we need to show $\eta$ is semi-simple. Let $M_{p} = \begin{bmatrix} \gamma \alpha^2 + \Id_{p} & \gamma \alpha \beta \\ - \gamma \alpha \beta  &  \gamma \alpha^2 \end{bmatrix}$, since $\alpha = \frac{\sqrt{5}}{5}$ is a $p$'th order root, we can simplify $M_p$ as:
\[
M_p = 
\sfrac{1}{5}
\begin{bmatrix} \pa{\gamma+5} \Id_{p} & {2\gamma} \Id_{p} \\ - {2\gamma} \Id_{p}  &  {\gamma} \Id_{p} \end{bmatrix}
\] 
As a result, we have
\[
\begin{aligned}
\rank\pa{M_{p} - \eta \Id_{2p}}
&= \rank\Ppa{ \small
\begin{bmatrix} \pa{\gamma+5-5\eta} \Id_{p} & {2\gamma} \Id_{p} \\ - {2\gamma} \Id_{p}  &  \pa{\gamma-5\eta} \Id_{p} \end{bmatrix} }
= p , \qandq \\
\rank\Pa{\pa{M_{p} - \eta \Id_{2p}}^2}
&= \rank\Ppa{ \small
\begin{bmatrix} (\pa{\gamma+5-5\eta}^2-4\gamma^2) \Id_{p} & {2\gamma}(\gamma+5-10\eta) \Id_{p} \\ - {2\gamma}(\gamma+5-10\eta) \Id_{p}  &  (\pa{\gamma-5\eta}^2-4\gamma^2) \Id_{p} \end{bmatrix} }
= p  ,
\end{aligned}
\]
which means $\eta$ is semi-simple by Definition \ref{def:semi_simple}, and we conclude the linear rate of convergence. 
\end{proof}

\begin{remark}
The proofs of the two propositions above is dimension independent, which means the results hold true for all puzzle sizes of perfect squares $s$ with $s \geq 4$. See Section \ref{sec:experiment} for numerical illustrations. 
\end{remark}

\subsection{$s$-queens puzzle}

The rule of eight queens puzzle is rather simple: placing eight chess queens on an $8\times8$ chessboard so that no two queens threaten each other. The size of puzzle can be generalized to any size $s\times s$ with $s \geq 4$, while there is no solution for $s=2,3$ and a trivial solution for $s=1$ which is obvious\footnote{\url{https://en.wikipedia.org/wiki/Eight_queens_puzzle}}.

We follow the setting of \cite{schaad2010modeling}. 
On the chessboard, as there are four directions (horizontal, vertical and two diagonal directions) for the queen to move, we have four constraint sets for the problem: 
\begin{itemize}
	\item $C_1$: each {\it row} has only one queen.
	\item $C_2$: each {\it column} has only one queen.
	\item $C_3$: each diagonal direction {\it southeast-northwest}, there is at most one queen.
	\item $C_4$: each diagonal direction {\it southwest-northeast}, there is at most one queen.
\end{itemize} 
Now we can formulate the $s$-queens puzzle as a feasibility problem of four sets
\beq\label{eq:queens_4set}
\find ~~x~ \in \bbR^{s\times s} \enskip \st \enskip  x \in C_1 \cap C_2 \cap C_3 \cap C_4 .
\eeq
Since all the sets above are binary, so is the set $\bcC \eqdef C_{1} \times \dotsm \times C_{4}$, as a result finite convergence can be obtained under the conditions of Theorems \ref{thm:local_dr} and \ref{thm:local_ddr}, for the standard Douglas--Rachford and the damped one, respectively.

\section{Numerical results} \label{sec:experiment}

We now provide numerical results on Sudoku and $s$-queens puzzles to support our theoretical findings. 
Before analyzing the convergence rates, we first compare the performance of the standard Douglas--Rachford splitting method \eqref{eq:dr} and the damped one \eqref{eq:ddr}, on how successful are they when applied to solve these two puzzles. \ie how often each method finds a feasible point.

The comparison is shown in Table~\ref{tb:cmp}. For \eqref{eq:ddr}, two choices of $\gamma$ are considered: $\gamma = \frac{1}{5} \in ]0, \sqrt{3/2}-1[$ suggested by Lemma \ref{thm:global_convergence} and $\gamma = 99$ with online tracking rule suggested in \cite[Remark 4]{li2016douglas}. For both methods, the iteration is terminated if either a stopping criterion is met or $10^4$ steps of iteration are reached, then we verify the output. Also, a minimal $100$ number of iteration is set. For a given puzzle, each method is repeated $10^3$ times with different initialization for each running.

For Sudoku, the size of both puzzles are $9\times9$: ``Puzzle 1'' is provided with $37$ digits, hence is easy; ``Puzzle 2'' has $22$ given digits and is more difficult than ``Puzzle 1''\footnote{For $9\times9$ Sudoku, to ensure the uniqueness of solution, the smallest number of given digits of a puzzle is $17$. See \url{https://www.technologyreview.com/2012/01/06/188520/mathematicians-solve-minimum-sudoku-problem/}}   
\begin{itemize}
\item The standard DR solves both puzzles with $100\%$ success rate, while dDR with $\gamma = \frac{1}{5}$ fails all tests. dDR with $\gamma = 99$ succeeds on ``Puzzle 1'' and the rate drops to about $88\%$ for ``Puzzle 2''.

\item In terms of number of iteration, sDR needs much less number of iterations compared to those of dDR with $\gamma = 99$. 
\end{itemize}
For $s$-queens puzzle, two different sizes are considered: $s=8$ for ``Puzzle 1'' and $s=16$ for ``Puzzle 2''. 
\begin{itemize}
\item Similar to Sudoku, dDR with $\gamma = \frac{1}{5}$ fails all tests. This time, between sDR and dDR with $\gamma=99$, neither achieves $100\%$ success rate with dDR being better than sDR. 

\item In terms of number of iteration, same as Sudoku case, sDR is better. 
\end{itemize}
The above observation, in particular the failure of dDR with $\gamma = \frac{1}{5}$, is in contrast to Example \ref{exp:circle_line}. 
One possible reason leads to the failure of dDR with $\gamma = \frac{1}{5}$, is that the set $C$ is finite and dDR can not escape bad local stationary point with small value of $\gamma$.

\renewcommand{\arraystretch}{1.1}

\begin{table}[!h]
\centering
\caption{Comparison of success rate of standard DR and damped DR for solving Sudoku and $s$-queens puzzles over 1,000 random initializations.}
\label{tb:cmp}
\begin{tabular}{|c|c|c|c|c|c|c|c|}
\hline
\multicolumn{2}{|c|}{\multirow{2}{*}{}} & \multicolumn{3}{c|}{Puzzle 1} & \multicolumn{3}{c|}{Puzzle 2} \\ \cline{3-8} 
\multicolumn{2}{|c|}{}                  &   sDR    &    dDR $\gamma = \frac{1}{5}$   &   dDR $\gamma = 99$    &   sDR   &   dDR $\gamma = \frac{1}{5}$   &   dDR $\gamma = 99$    \\ \hline
\multirow{2}{*}{Sudoku}           &  success rate    &   100\%    &    0   &    100\%   &   100\%    &   0    &   89.7\%   \\ \cline{2-8} 
                            &   avg. \# of itr.     &    114   &   184    &   2710    &   408    &   184    &   5409    \\ \hline
\multirow{2}{*}{$s$-queens}           &  success rate   &    94.8\%   &   0    &   98.0\%    &    90.2\%   &   0    & 92.2\%       \\ \cline{2-8} 
                            &   avg. \# of itr.   &  653  &  100  &  2812   &   1286  &  100   &  3618   \\ \hline
\end{tabular}
\end{table}

\renewcommand{\arraystretch}{1.0}


\subsection{Sudoku puzzle}\label{sec:exp_sudoku}

We consider three different puzzle sizes for Sudoku to verify out results: $4$, $9$ and $16$, which are shown in Figure~\ref{fig:puzzle_S} (a)-(c). In each size, we have $4$, $32$, and $128$ coefficients provided respectively.
The convergence behavior of standard Douglas--Rachford splitting method can be seen in the second and third rows of Figure \ref{fig:puzzle_S}, from which we observe that for all puzzles,
\begin{itemize}
	\item {\it Finite termination} of $u_{i,k}, i=1,\ldots,4$: in the second row of Figure \ref{fig:puzzle_S}, we provide the $\ell_0$ pseudo-norms of $\norm{u_{i,k}-u_{i}^\star}_{0}, i=1,\ldots,4$ to show the mismatch between $u_{i,k}$ and $u_{i}^\star$. We observed that, for each $i \in \ba{1,2,3,4}$, $\norm{u_{i,k}-u_{i}^\star}_{0}$ reaches $0$ in finite steps, which means the finite termination. 
	\item {\it Local linear convergence} In the last row of Figure \ref{fig:puzzle_S}, we provide the convergence behaviors of $\norm{\bmu_k-\bmu^\star}$ (which actually reduces to $\norm{u_{5,k}-u_{5}^\star}$), $\norm{\xk-\xsol}$ and $\norm{\bmz_k-\bmz^\star}$. Take $\norm{\bmz_k-\bmz^\star}$ for example, its convergence has two different regimes: {\it sub-linear} rate from the beginning, and {\it linear} rate locally. The {\it magenta} dashed line is our theoretical estimation of the linear convergence rate and the slope of the line is $\frac{\sqrt{5}}{5}$. 
\end{itemize}
For all three different puzzle sizes, the local linear convergence rate is $\frac{\sqrt{5}}{5}\approx0.45$, which confirms that the rate is independent of puzzle size.

\begin{figure}[!ht]
	\centering
	\subfloat[Size $4\times4$]{ \includegraphics[width=0.31\linewidth]{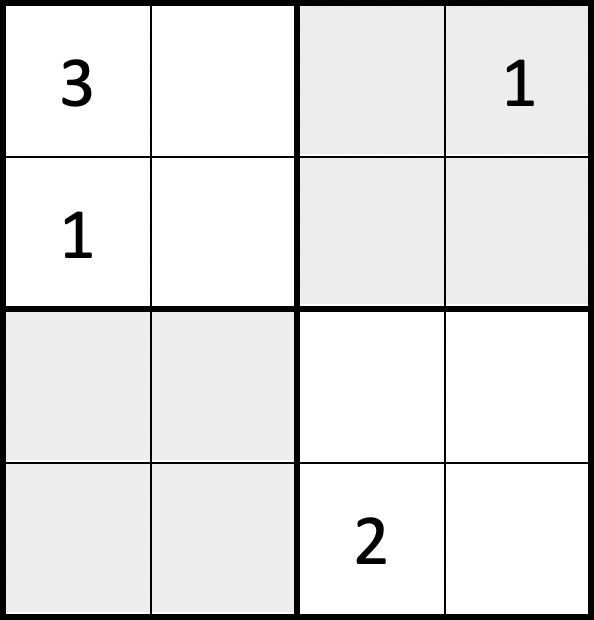} } \hspace{-2pt}
	\subfloat[Size $9\times9$]{ \includegraphics[width=0.312\linewidth]{puzzle_9.png} } \hspace{-2pt}
	\subfloat[Size $16\times16$]{ \includegraphics[width=0.314\linewidth]{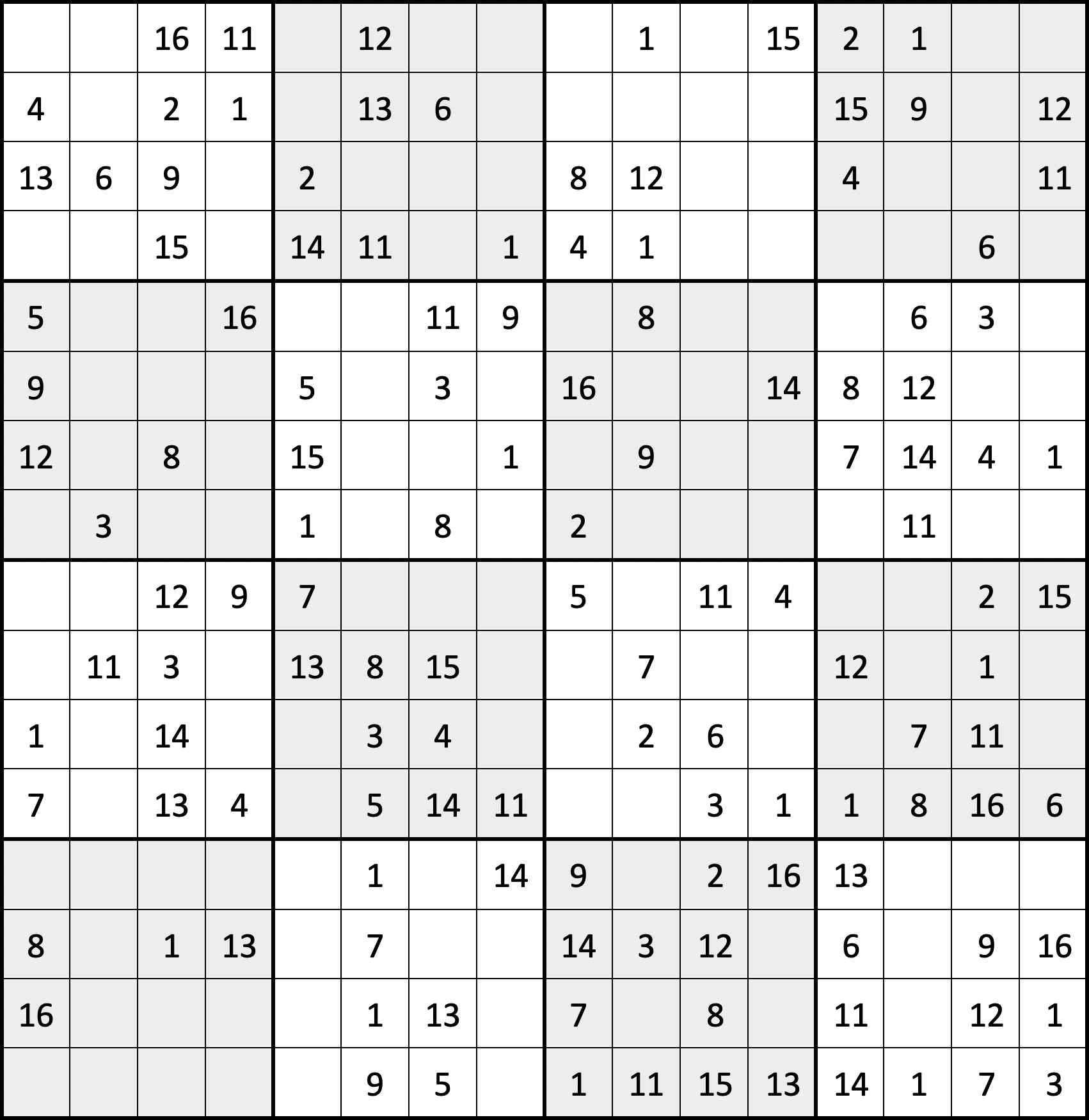} } \\[-2mm]
	\subfloat[Size $4$: convergence of $u_{i,k}$]{ \includegraphics[width=0.31\linewidth]{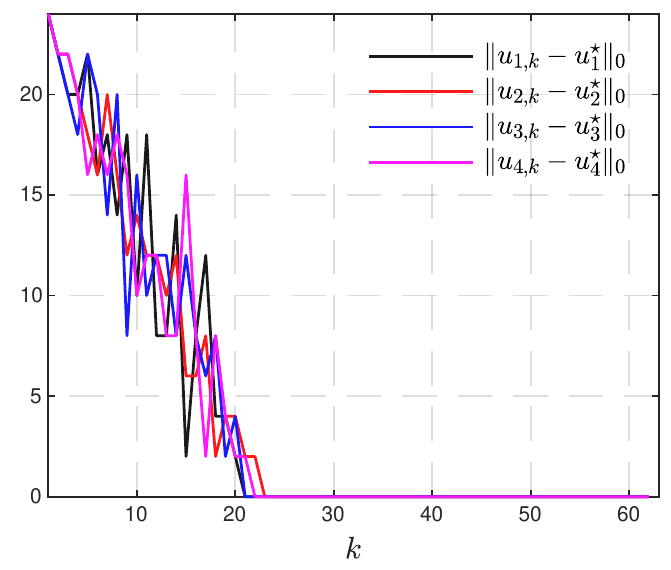} } \hspace{0pt}
	\subfloat[Size $9$: convergence of $u_{i,k}$]{ \includegraphics[width=0.31\linewidth]{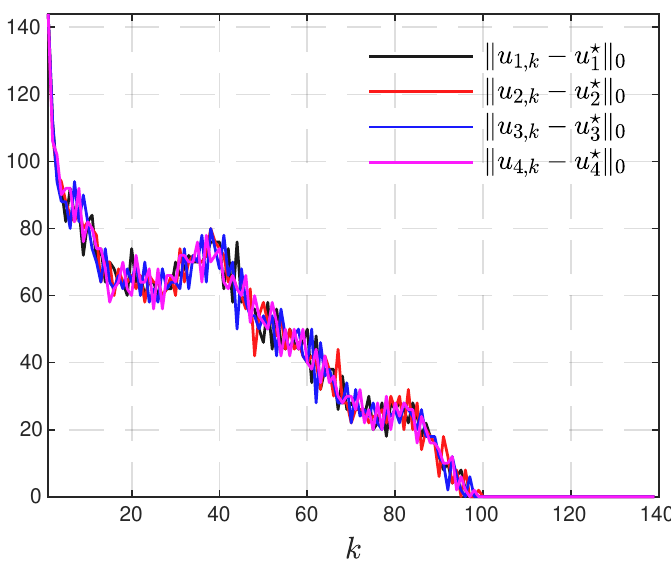} } \hspace{0pt}
	\subfloat[Size $16$: convergence of $u_{i,k}$]{ \includegraphics[width=0.31\linewidth]{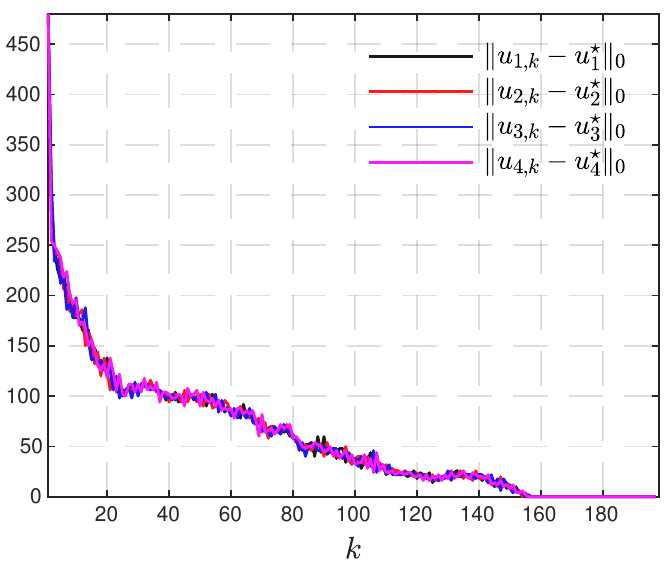} } \\[-2mm]
	\subfloat[Size $4$: convergence of $(\bmu_k,x_k,\bmz_k)$]{ \includegraphics[width=0.31\linewidth]{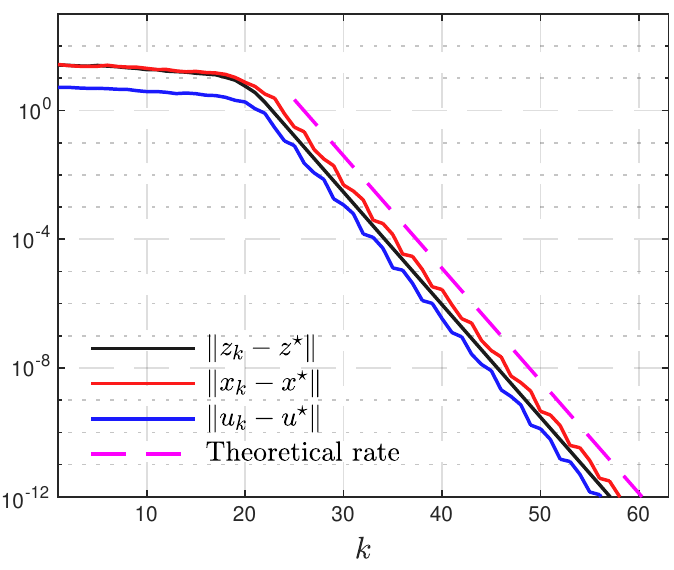} } \hspace{0pt}
	\subfloat[Size $9$: convergence of $(\bmu_k,x_k,\bmz_k)$]{ \includegraphics[width=0.31\linewidth]{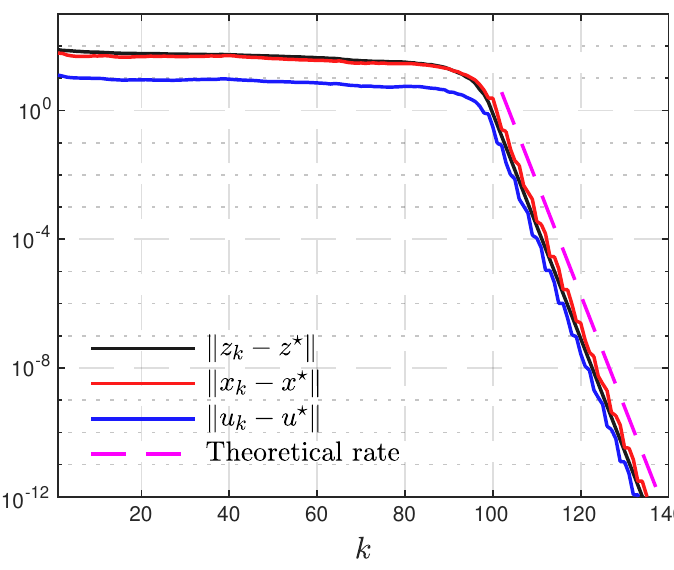} } \hspace{0pt}
	\subfloat[Size $16$: convergence of $(\bmu_k,x_k,\bmz_k)$]{ \includegraphics[width=0.31\linewidth]{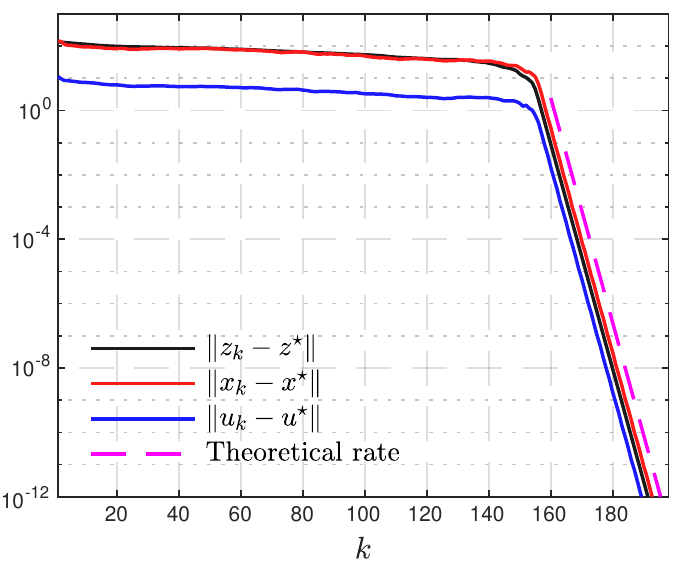} } \\
	\caption{Different sizes of Sudoku puzzles and convergence observations.} 
	\label{fig:puzzle_S}
\end{figure}

\subsection{$s$-queens puzzle}

For the $s$-queens puzzle, we also consider three different puzzle sizes: $s = 8, 16$ and $25$, which are shown in Figure \ref{fig:cmp_queens} (a)-(c). 
The convergence behaviors of the Douglas--Rachford splitting method are shown in the second row of Figure \ref{fig:cmp_queens}. Since all the constraint sets are binary, we observe finite convergence for the algorithm which complies with our theoretical results. 
%

\begin{figure}[!ht]
	\centering
	\subfloat[Size 8]{ \includegraphics[width=0.31\linewidth]{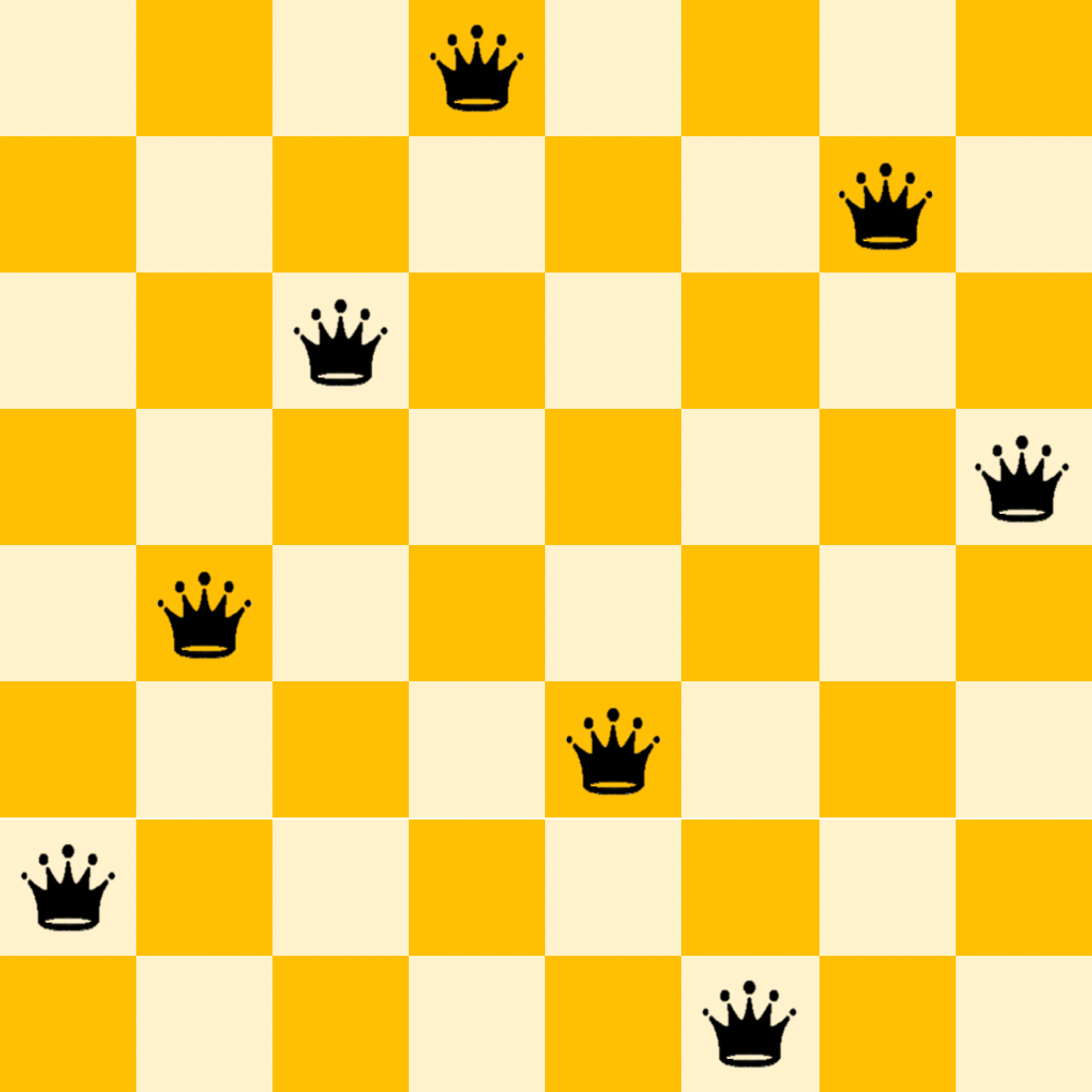} } \hspace{-2pt}
	\subfloat[Size 16]{ \includegraphics[width=0.31\linewidth]{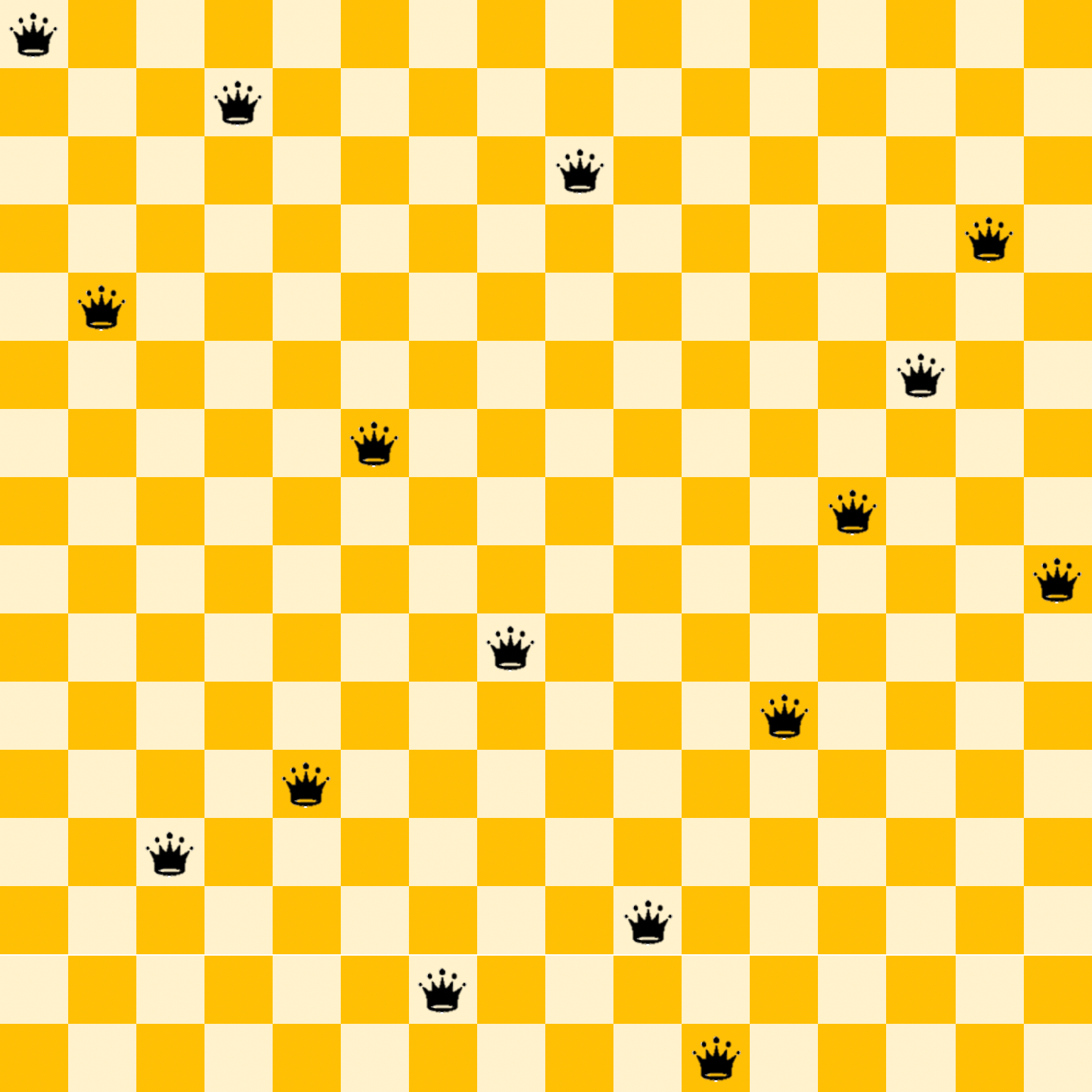} } \hspace{-2pt}
	\subfloat[Size 25]{ \includegraphics[width=0.322\linewidth]{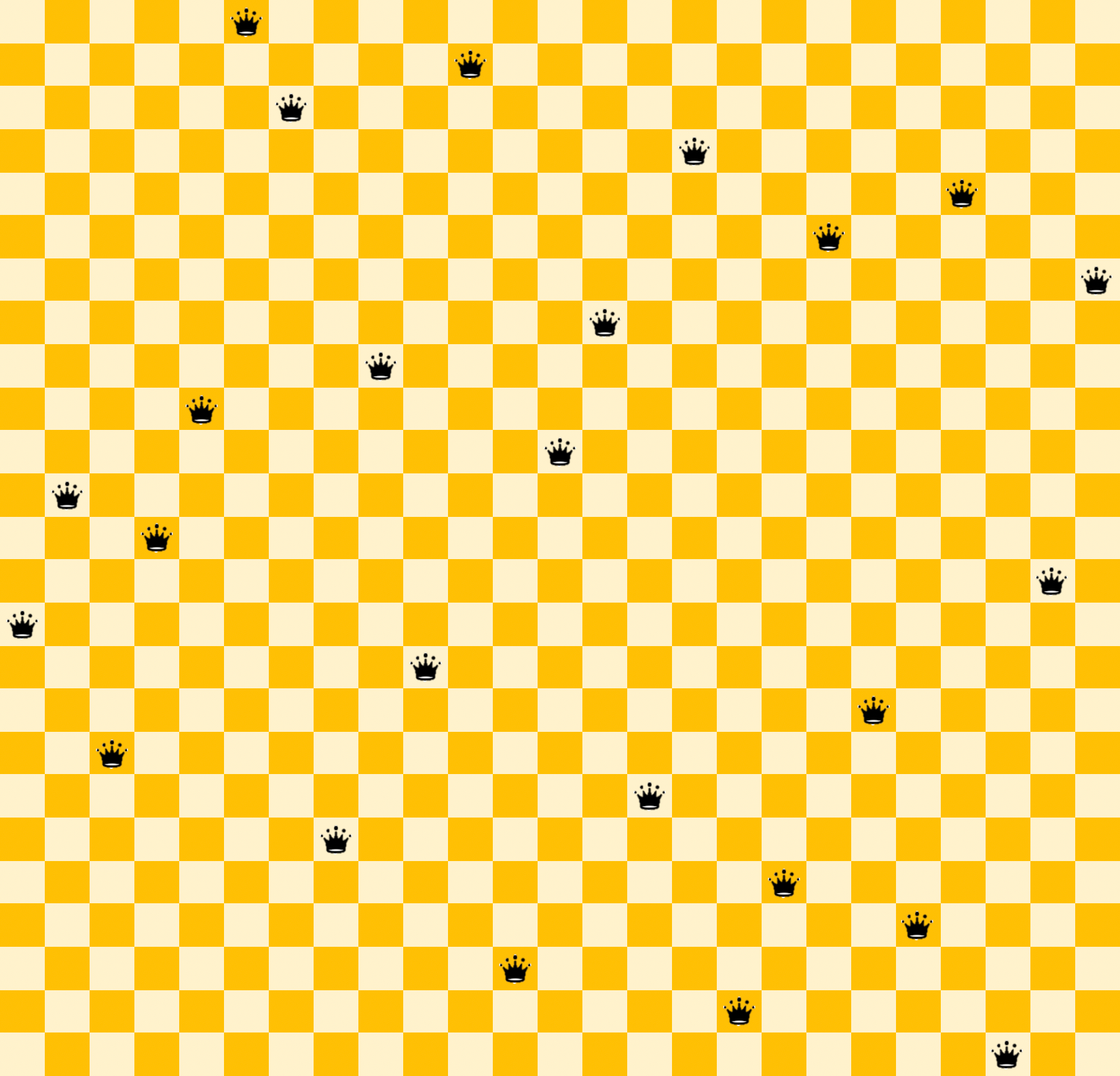} } \\
	\subfloat[Size 8]{ \includegraphics[width=0.31\linewidth]{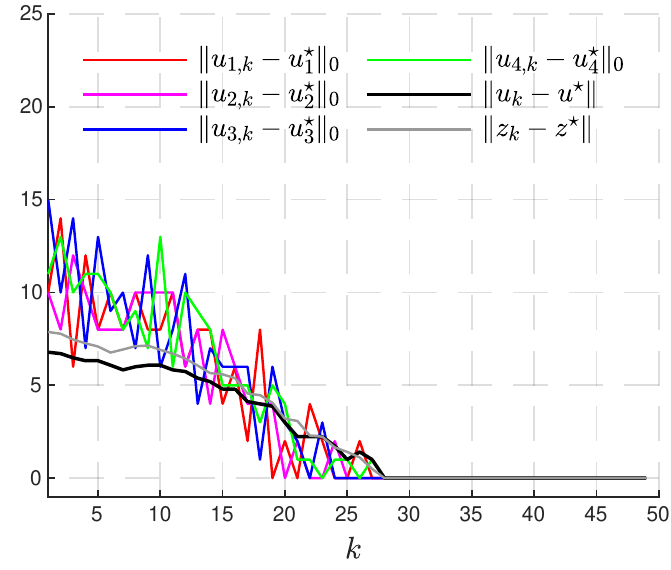} } \hspace{0pt}
	\subfloat[Size 16]{ \includegraphics[width=0.31\linewidth]{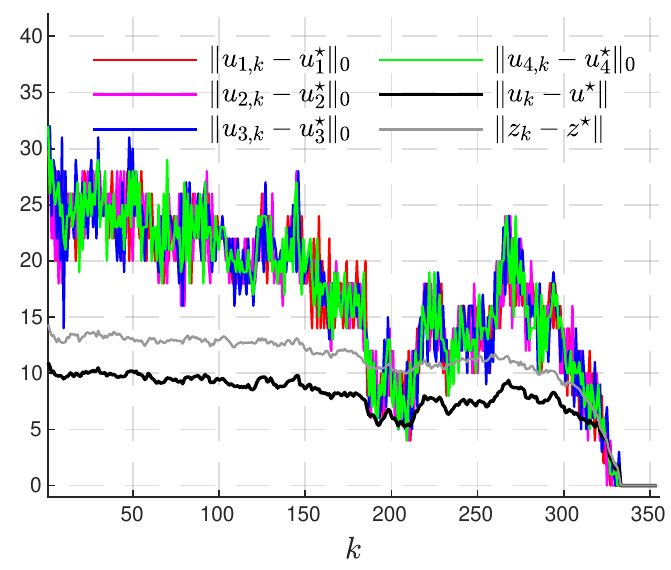} } \hspace{0pt}
	\subfloat[Size 25]{ \includegraphics[width=0.31\linewidth]{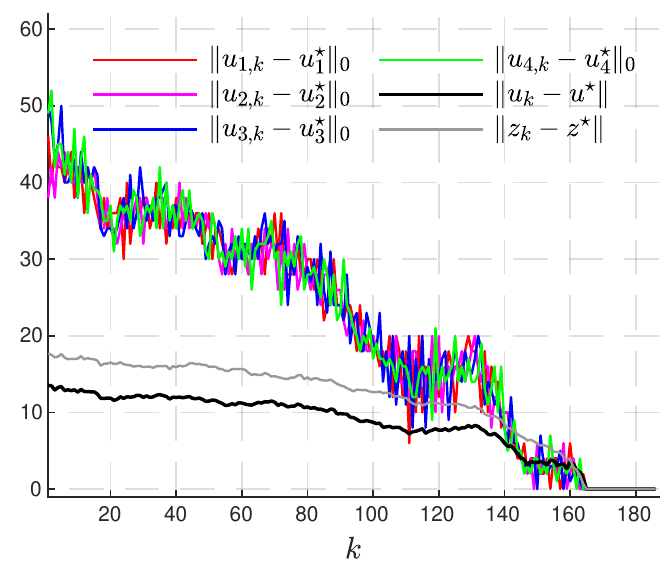} } \\
	\caption{Different sizes of queens puzzles and convergence observations.} 
	\label{fig:cmp_queens}
\end{figure}

\subsection{The damped Douglas--Rachford splitting}

We conclude our numerical experiments by showing the local linear convergence the damped Douglas--Rachford splitting method with $\gamma = 99$. 
The results on Sudoku puzzle of size $9\times9$ and eight queens puzzle of size $8\times8$ are shown below in Figure \ref{fig:cmp_ddr}. For both plots, the {\it magenta} line is our theoretical estimation of the local linear rate:
\begin{itemize}
\item For Sudoku puzzle, the slope of the magenta line is $\frac{ {2\gamma + 5} + \sqrt{25 - 16\gamma^2} }{ 10(1+\gamma) } \approx 0.86$.
\item For eight queens puzzle, the slope of the magenta line is $\frac{ \gamma }{ 1+\gamma } \approx 0.17$.
\end{itemize}
Again, our theoretical estimations are tight. 
We omit the plots of dDR with $\gamma = 99$ as they are very similar to those of Figure \ref{fig:cmp_ddr}, except different rates of local linear convergence.

\begin{figure}[!ht]
	\centering
	\subfloat[Sudoku puzzle]{ \includegraphics[width=0.4\linewidth]{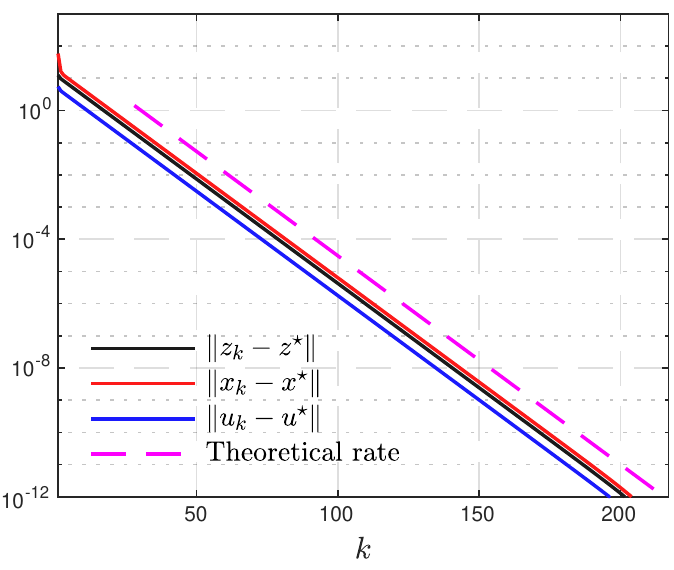} } \hspace{4pt}
	\subfloat[Eight queens puzzle]{ \includegraphics[width=0.4\linewidth]{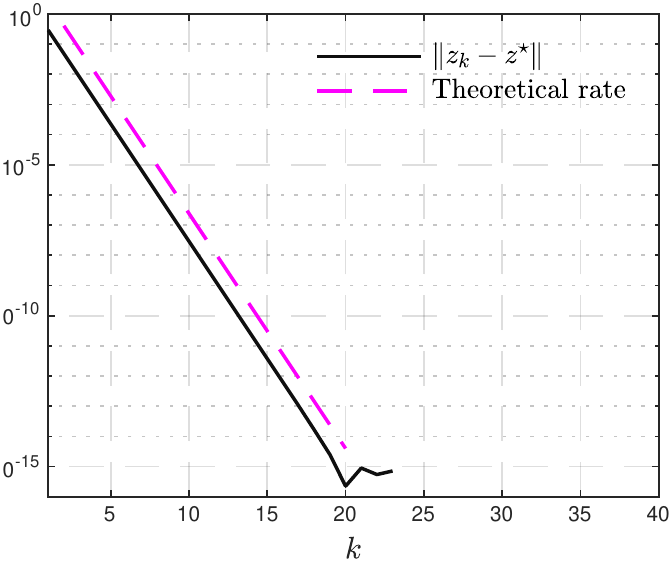} } \\
	\caption{Local linear convergence of the damped Douglas-Rachford for Sudoku puzzle and eight queens puzzle. Note again that convergence does not imply finding a solution to the feasibility problem.} 
	\label{fig:cmp_ddr}
\end{figure}

\section{Conclusions}

In this paper, we studied local convergence properties of Douglas--Rachford splitting method when applied to solve non-convex feasibility problems. 
Under a proper non-degeneracy condition, both finite convergence and local linear convergence are proved for the standard Douglas--Rachford splitting and a damped version of the method. 
Understanding when the methods fail, especially for the damped Douglas--Rachford splitting, require further study on the property of the methods.

\paragraph{Acknowledgement}
We would like to thank Guoyin Li for helpful discussions on the convergence of Douglas--Rachford splitting for non-convex optimization. 
J.L. was partly supported by Leverhulme trust, Newton trust and the EPSRC centre ``EP/N014588/1''. R.T. acknowledges funding from EPSRC Grant No. ``EP/L016516/1'' for the Cambridge Centre for Analysis. 
Both authors were supported by the Cantab Capital Institute for Mathematics of Information.

\begin{small}
\bibliographystyle{plain}
\bibliography{bib}
\end{small}

\end{document}